\documentclass[12pt]{article}
\usepackage{graphicx} 
\usepackage{geometry}
\usepackage[dvipsnames,svgnames]{xcolor}

\usepackage{tcolorbox}
\usepackage{amssymb}
\usepackage{amsmath}
\usepackage{amsthm}
\usepackage{mathrsfs}
\usepackage{color}
\usepackage{hyperref}
\usepackage{ytableau}
\usepackage{tikz} 
\usepackage{graphicx}
\usepackage{rotating} 

\usepackage{pict2e}

\newtheorem{theorem}{Theorem}[section]
\newtheorem{conjecture}[theorem]{Conjecture}

\newtheorem{proposition}[theorem]{ Proposition}
\newtheorem{corollary}[theorem]{ Corollary}

\newtheorem{problem}[theorem]{Problem}

\newtheorem{remark}[theorem]{{Remark}}

\usepackage{graphicx} 

\usepackage{xcolor}

\newcommand{\BPD}[2][1pc]{%
\setlength{\unitlength}{#1}
\def\BPDframe{%
    \thinlines%
    \color{lightgray}%
    \put(0,0){\line(0,1){1}}%
    \put(1,0){\line(0,1){1}}%
    \put(0,0){\line(1,0){1}}%
    \put(0,1){\line(1,0){1}}%
    \linethickness{0.08\unitlength}%
    \color{teal}}
\def\O{
\begin{picture}(1,1)
    \BPDframe
\end{picture}}
\def\delO{%
    \begin{picture}(1,1)
    \BPDframe
    \color{red}
    \qbezier(0,0)(0,1)(0,1)
    \qbezier(0,1)(1,1)(1,1)
    \qbezier(1,1)(1,0)(1,0)
    \qbezier(1,0)(0,0)(0,0)
    \end{picture}%
}
\def\MB{%
   \begin{picture}(1,1)
      \BPDframe
      \color{orange}
    \qbezier(0.5,0)(0.5,0.2)(0.5,0.2)
    \qbezier(1,0.5)(0.8,0.5)(0.8,0.5)
    \qbezier(0.8,0.5)(0.5,0.5)(0.5,0.2)
    \qbezier(0.5,1)(0.5,0.8)(0.5,0.8)
    \qbezier(0,0.5)(0.2,0.5)(0.2,0.5)
    \qbezier(0.5,0.8)(0.5,0.5)(0.2,0.5)
   \end{picture}%
}
\def\BoldMBup{%
   \begin{picture}(1,1)
      \BPDframe
      \color{orange}
    \qbezier(0.5,0)(0.5,0.2)(0.5,0.2)
    \qbezier(1,0.5)(0.8,0.5)(0.8,0.5)
    \qbezier(0.8,0.5)(0.5,0.5)(0.5,0.2)
      \linethickness{2pt}
    \qbezier(0.5,1)(0.5,0.8)(0.5,0.8)
    \qbezier(0,0.5)(0.2,0.5)(0.2,0.5)
    \qbezier(0.5,0.8)(0.5,0.5)(0.2,0.5)
   \end{picture}%
}\def\BoldMBdown{%
   \begin{picture}(1,1)
      \BPDframe
     \color{orange}
    \qbezier(0.5,1)(0.5,0.8)(0.5,0.8)
    \qbezier(0,0.5)(0.2,0.5)(0.2,0.5)
    \qbezier(0.5,0.8)(0.5,0.5)(0.2,0.5)
      \linethickness{2pt}
    \qbezier(0.5,0)(0.5,0.2)(0.5,0.2)
    \qbezier(1,0.5)(0.8,0.5)(0.8,0.5)
    \qbezier(0.8,0.5)(0.5,0.5)(0.5,0.2)
   \end{picture}%
}
\def\delF{%
    \begin{picture}(1,1)
    \BPDframe
    \color{red}
    \qbezier(0.5,0)(0.5,0.2)(0.5,0.2)
    \qbezier(1,0.5)(0.8,0.5)(0.8,0.5)
    \qbezier(0.8,0.5)(0.5,0.5)(0.5,0.2)
    \qbezier(0,0)(0.5,0.5)(1,1)
    \end{picture}%
}
\def\delFF{%
    \begin{picture}(1,1)
    \BPDframe
    \qbezier(0.5,0)(0.5,0.2)(0.5,0.2)
    \qbezier(1,0.5)(0.8,0.5)(0.8,0.5)
    \qbezier(0.8,0.5)(0.5,0.5)(0.5,0.2)
    \color{red}
    \qbezier(0,0)(0.5,0.5)(1,1)
    \end{picture}%
}
\def\F{
\begin{picture}(1,1)
    \BPDframe
    \qbezier(0.5,0)(0.5,0.2)(0.5,0.2)
    \qbezier(1,0.5)(0.8,0.5)(0.8,0.5)
    \qbezier(0.8,0.5)(0.5,0.5)(0.5,0.2)
\end{picture}}
\def\J{
\begin{picture}(1,1)
    \BPDframe
    \qbezier(0.5,1)(0.5,0.8)(0.5,0.8)
    \qbezier(0,0.5)(0.2,0.5)(0.2,0.5)
    \qbezier(0.5,0.8)(0.5,0.5)(0.2,0.5)
\end{picture}}
\def\delJ{%
    \begin{picture}(1,1)
    \BPDframe
    \color{red}
    \qbezier(0.5,1)(0.5,0.8)(0.5,0.8)
    \qbezier(0,0.5)(0.2,0.5)(0.2,0.5)
    \qbezier(0.5,0.8)(0.5,0.5)(0.2,0.5)
    \qbezier(0,0)(0.5,0.5)(1,1)
    \end{picture}%
}
\def\delJJ{%
    \begin{picture}(1,1)
    \BPDframe
    \qbezier(0.5,1)(0.5,0.8)(0.5,0.8)
    \qbezier(0,0.5)(0.2,0.5)(0.2,0.5)
    \qbezier(0.5,0.8)(0.5,0.5)(0.2,0.5)
    \color{red}
    \qbezier(0,0)(0.5,0.5)(1,1)
    \end{picture}%
}
\def\Z{
\begin{picture}(1,1)
    \BPDframe
    \qbezier(0.5,0)(0.5,0.2)(0.5,0.2)
    \qbezier(0,0.5)(0.2,0.5)(0.2,0.5)
    \qbezier(0.2,0.5)(0.5,0.5)(0.5,0.2)
\end{picture}}
\def\L{
\begin{picture}(1,1)
    \BPDframe
    \qbezier(0.5,1)(0.5,0.8)(0.5,0.8)
    \qbezier(1,0.5)(0.8,0.5)(0.8,0.5)
    \qbezier(0.5,0.8)(0.5,0.5)(0.8,0.5)
\end{picture}}
\def\I{
\begin{picture}(1,1)
    \BPDframe
    \qbezier(0.5,0)(0.5,0.5)(0.5,1)
\end{picture}}
\def\H{
\begin{picture}(1,1)
    \BPDframe
    \qbezier(0,0.5)(0.5,0.5)(1,0.5)
\end{picture}}
\def\x{
\begin{picture}(1,1)
    \BPDframe
    \qbezier(0,0.5)(0.5,0.5)(1,0.5)
    \qbezier(0.5,0)(0.5,0.3)(0.5,0.3)
    \qbezier(0.5,1)(0.5,0.7)(0.5,0.7)
\end{picture}}
\def\X{
\begin{picture}(1,1)
    \BPDframe
    \qbezier(0.5,0)(0.5,0.5)(0.5,1)
    \qbezier(0,0.5)(0.3,0.5)(0.3,0.5)
    \qbezier(1,0.5)(0.7,0.5)(0.7,0.5)
\end{picture}}
\def\grayX{
\begin{picture}(1,1)
\put(0,0){\color{gray!30}\rule{\unitlength}{\unitlength}}%
    \BPDframe
    \qbezier(0.5,0)(0.5,0.5)(0.5,1)
    \qbezier(0,0.5)(0.3,0.5)(0.3,0.5)
    \qbezier(1,0.5)(0.7,0.5)(0.7,0.5)
\end{picture}}
\def\delXAboveX{%
    \begin{picture}(1,1)
    \BPDframe
    \color{yellow}
    \qbezier(0.5,0)(0.5,0.5)(0.5,1)
    \qbezier(0,0.5)(0.3,0.5)(0.3,0.5)
    \qbezier(1,0.5)(0.7,0.5)(0.7,0.5)
    \qbezier(0,0)(0,1)(0,1)
    \qbezier(0,1)(1,1)(1,1)
    \qbezier(1,1)(1,0)(1,0)
    \qbezier(1,0)(0,0)(0,0)
    \end{picture}%
}
\def\blueX{%
    \begin{picture}(1,1)
    \BPDframe
    \color{blue}
    \qbezier(0.5,0)(0.5,0.5)(0.5,1)
    \qbezier(0,0.5)(0.3,0.5)(0.3,0.5)
    \qbezier(1,0.5)(0.7,0.5)(0.7,0.5)
    \qbezier(0,0)(0,1)(0,1)
    \qbezier(0,1)(1,1)(1,1)
    \qbezier(1,1)(1,0)(1,0)
    \qbezier(1,0)(0,0)(0,0)
    \end{picture}%
}
\def\delX{%
    \begin{picture}(1,1)
    \BPDframe
    \qbezier(0.5,0)(0.5,0.5)(0.5,1)
    \qbezier(0,0.5)(0.3,0.5)(0.3,0.5)
    \qbezier(1,0.5)(0.7,0.5)(0.7,0.5)
    \color{red}
    \qbezier(0,0)(0,1)(0,1)
    \qbezier(0,1)(1,1)(1,1)
    \qbezier(1,1)(1,0)(1,0)
    \qbezier(1,0)(0,0)(0,0)
    \end{picture}%
}
\def\BoldX{%
    \begin{picture}(1,1)
    \BPDframe
    \qbezier(0.5,0)(0.5,0.5)(0.5,1)
    \linethickness{2pt}
    \qbezier(0,0.5)(0.3,0.5)(0.3,0.5)
    \qbezier(1,0.5)(0.7,0.5)(0.7,0.5)
    \end{picture}%
}
\def\BolddelX{%
    \begin{picture}(1,1)
    \BPDframe
    \qbezier(0.5,0)(0.5,0.5)(0.5,1)
    \color{red}
    \qbezier(0,0)(0,1)(0,1)
    \qbezier(0,1)(1,1)(1,1)
    \qbezier(1,1)(1,0)(1,0)
    \qbezier(1,0)(0,0)(0,0)
    \color{teal}
    \linethickness{2pt}
    \qbezier(0,0.5)(0.3,0.5)(0.3,0.5)
    \qbezier(1,0.5)(0.7,0.5)(0.7,0.5)
    \end{picture}%
}
\def\B{
\begin{picture}(1,1)
    \BPDframe
    \qbezier(0.5,0)(0.5,0.2)(0.5,0.2)
    \qbezier(1,0.5)(0.8,0.5)(0.8,0.5)
    \qbezier(0.8,0.5)(0.5,0.5)(0.5,0.2)
    \qbezier(0.5,1)(0.5,0.8)(0.5,0.8)
    \qbezier(0,0.5)(0.2,0.5)(0.2,0.5)
    \qbezier(0.5,0.8)(0.5,0.5)(0.2,0.5)
\end{picture}}
\def\delBUP{%
    \begin{picture}(1,1)
    \BPDframe
    \qbezier(0.5,0)(0.5,0.2)(0.5,0.2)
    \qbezier(1,0.5)(0.8,0.5)(0.8,0.5)
    \qbezier(0.8,0.5)(0.5,0.5)(0.5,0.2)
    \qbezier(0.5,1)(0.5,0.8)(0.5,0.8)
    \qbezier(0,0.5)(0.2,0.5)(0.2,0.5)
    \qbezier(0.5,0.8)(0.5,0.5)(0.2,0.5)
    \color{red}
    \qbezier(0,0)(0.5,0.5)(1,1)
    \end{picture}%
}
\def\delBDown{%
    \begin{picture}(1,1)
    \BPDframe
    \qbezier(0.5,1)(0.5,0.8)(0.5,0.8)
    \qbezier(0,0.5)(0.2,0.5)(0.2,0.5)
    \qbezier(0.5,0.8)(0.5,0.5)(0.2,0.5)
    \qbezier(0.5,0)(0.5,0.2)(0.5,0.2)
    \qbezier(1,0.5)(0.8,0.5)(0.8,0.5)
    \qbezier(0.8,0.5)(0.5,0.5)(0.5,0.2)
    \color{red}
    \qbezier(0,0)(0.5,0.5)(1,1)
    \end{picture}%
}
\def\delB{%
    \begin{picture}(1,1)
    \BPDframe
    \qbezier(0.5,1)(0.5,0.8)(0.5,0.8)
    \qbezier(0,0.5)(0.2,0.5)(0.2,0.5)
    \qbezier(0.5,0.8)(0.5,0.5)(0.2,0.5)
    \qbezier(0.5,0)(0.5,0.2)(0.5,0.2)
    \qbezier(1,0.5)(0.8,0.5)(0.8,0.5)
    \qbezier(0.8,0.5)(0.5,0.5)(0.5,0.2)
    \color{red}
    \qbezier(0,0)(0,1)(0,1)
    \qbezier(0,1)(1,1)(1,1)
    \qbezier(1,1)(1,0)(1,0)
    \qbezier(1,0)(0,0)(0,0)
    \end{picture}%
}
\def\BoldupdelB{%
    \begin{picture}(1,1)
    \BPDframe
    \qbezier(0.5,0)(0.5,0.2)(0.5,0.2)
    \qbezier(1,0.5)(0.8,0.5)(0.8,0.5)
    \qbezier(0.8,0.5)(0.5,0.5)(0.5,0.2)
    \color{red}
    \qbezier(0,0)(0,1)(0,1)
    \qbezier(0,1)(1,1)(1,1)
    \qbezier(1,1)(1,0)(1,0)
    \qbezier(1,0)(0,0)(0,0)
    \color{teal}
    \linethickness{2pt}
    \qbezier(0.5,1)(0.5,0.8)(0.5,0.8)
    \qbezier(0,0.5)(0.2,0.5)(0.2,0.5)
    \qbezier(0.5,0.8)(0.5,0.5)(0.2,0.5)
    \end{picture}%
}
\def\BolddowndelB{%
    \begin{picture}(1,1)
    \BPDframe
    \qbezier(0.5,1)(0.5,0.8)(0.5,0.8)
    \qbezier(0,0.5)(0.2,0.5)(0.2,0.5)
    \qbezier(0.5,0.8)(0.5,0.5)(0.2,0.5)
    \color{red}
    \qbezier(0,0)(0,1)(0,1)
    \qbezier(0,1)(1,1)(1,1)
    \qbezier(1,1)(1,0)(1,0)
    \qbezier(1,0)(0,0)(0,0)
    \color{teal}
    \linethickness{2pt}
    \qbezier(0.5,0)(0.5,0.2)(0.5,0.2)
    \qbezier(1,0.5)(0.8,0.5)(0.8,0.5)
    \qbezier(0.8,0.5)(0.5,0.5)(0.5,0.2)
    \end{picture}%
}
\def\BoldBUP{%
    \begin{picture}(1,1)
    \BPDframe
    \qbezier(0.5,0)(0.5,0.2)(0.5,0.2)
    \qbezier(1,0.5)(0.8,0.5)(0.8,0.5)
    \qbezier(0.8,0.5)(0.5,0.5)(0.5,0.2)
    \linethickness{2pt}
    \qbezier(0.5,1)(0.5,0.8)(0.5,0.8)
    \qbezier(0,0.5)(0.2,0.5)(0.2,0.5)
    \qbezier(0.5,0.8)(0.5,0.5)(0.2,0.5)
    \end{picture}%
}
\def\BoldBDown{%
    \begin{picture}(1,1)
    \BPDframe
    \qbezier(0.5,1)(0.5,0.8)(0.5,0.8)
    \qbezier(0,0.5)(0.2,0.5)(0.2,0.5)
    \qbezier(0.5,0.8)(0.5,0.5)(0.2,0.5)
    \linethickness{2pt}
    \qbezier(0.5,0)(0.5,0.2)(0.5,0.2)
    \qbezier(1,0.5)(0.8,0.5)(0.8,0.5)
    \qbezier(0.8,0.5)(0.5,0.5)(0.5,0.2)
    \end{picture}%
}
\def\BolddelBUP{%
    \begin{picture}(1,1)
    \BPDframe
    \qbezier(0.5,0)(0.5,0.2)(0.5,0.2)
    \qbezier(1,0.5)(0.8,0.5)(0.8,0.5)
    \qbezier(0.8,0.5)(0.5,0.5)(0.5,0.2)
    \color{red}
    \qbezier(0,0)(0.5,0.5)(1,1)
    \color{teal}
    \linethickness{2pt}
    \qbezier(0.5,1)(0.5,0.8)(0.5,0.8)
    \qbezier(0,0.5)(0.2,0.5)(0.2,0.5)
    \qbezier(0.5,0.8)(0.5,0.5)(0.2,0.5)
    \end{picture}%
}
\def\BolddelBDown{%
    \begin{picture}(1,1)
    \BPDframe
     \qbezier(0.5,1)(0.5,0.8)(0.5,0.8)
    \qbezier(0,0.5)(0.2,0.5)(0.2,0.5)
    \qbezier(0.5,0.8)(0.5,0.5)(0.2,0.5)
    \color{red}
    \qbezier(0,0)(0.5,0.5)(1,1)
    \linethickness{2pt}
    \color{teal}
    \qbezier(0.5,0)(0.5,0.2)(0.5,0.2)
    \qbezier(1,0.5)(0.8,0.5)(0.8,0.5)
    \qbezier(0.8,0.5)(0.5,0.5)(0.5,0.2)
    \end{picture}%
}
\def\b{
\begin{picture}(1,1)
    \BPDframe
    \qbezier(0.5,0)(0.5,0.2)(0.5,0.2)
    \qbezier(0,0.5)(0.2,0.5)(0.2,0.5)
    \qbezier(0.2,0.5)(0.5,0.5)(0.5,0.2)
    \qbezier(0.5,1)(0.5,0.8)(0.5,0.8)
    \qbezier(1,0.5)(0.8,0.5)(0.8,0.5)
    \qbezier(0.5,0.8)(0.5,0.5)(0.8,0.5) 
\end{picture}}
\def\C{
\begin{picture}(1,1)
    \BPDframe
    \qbezier(0,0)(0.5,0.5)(1,1)
    \qbezier(1,0)(0.5,0.5)(0,1)
\end{picture}}
\def\S{
\begin{picture}(1,1)
    \BPDframe
    1
\end{picture}}
\def\M##1{\begin{picture}(1,1)%
    
    \put(0,0.2){\makebox[\unitlength]{\(##1\)}}
\end{picture}}
\def\Lb##1{%
   \begin{picture}(1,1)
      \BPDframe
      \put(0,0.2){\makebox[\unitlength]{\(##1\)}}
   \end{picture}%
}
\begin{array}{@{\,}c@{\,}}{}
{\def\arraystretch{0}
\setlength{\arraycolsep}{0pc}
\color{teal}
\begin{array}{@{}l@{}}%
#2\end{array}}
\end{array}}

\begin{document}

 \begin{center}
{\Large\bf Principal specializations of Grothendieck polynomials}

\vskip 6mm
{\small   }
Haojun Bai, Feng Gu,  Peter L. Guo, Jiaji Liu

\end{center}

\begin{abstract}
Motivated by Stanley's ``Schubert shenanigans'' paper,  commendable attempts have been made to understand the principal specializations of Schubert or Grothendieck polynomials. In this paper, we prove that when a  permutation $w$ does not contain  the $1423$ pattern, the principal specialization of the corresponding   $\beta$-Grothendieck polynomial can be expressed nonnegatively in terms of the occurrences of patterns in $w$. Using  an inverse conservation principle, we further obtain  the nonnegativity expansion for permutations avoiding the $1342$ pattern. Our results  partially resolve   conjectures raised  respectively by Gao (independently observed  by Gaetz),  Me\'sz\'aros--Tanjaya, and Dennin. The proofs are achieved  based upon a reduction algorithm performing on  the classic  pipe dream model of $\beta$-Grothendieck polynomials. 
\end{abstract}

\section{Introduction}

The
Schubert polynomial  $\mathfrak{S}_w(x)$, indexed by a permutation $w$,  was  introduced by   Lascoux  and   Schützenberger \cite{lascoux1989fonctorialite}, representing  the Schubert class $[X_w]$ in  the  cohomology ring of the    flag variety. Its $K$-theory counterpart, called the  Grothendieck polynomial $\mathfrak{G}_w(x)$ \cite{LS-Gro}, is the polynomial representative of the Schubert class in the Grothendieck ring of the    flag variety. Fomin and Kirillov   \cite{fomin1994grothendieck} defined the $\beta$-Grothendieck polynomial $\mathfrak{G}_w^{(\beta)}(x)$  which specializes to $\mathfrak{S}_w(x)$ and $\mathfrak{G}_w(x)$ by taking $\beta=0$ and $\beta=-1$, respectively. Schubert/Grothendieck polynomials are pivotal objects in the study of Schubert calculus for the flag variety, and admit a rich collection of combinatorial models: including for example the classic pipe dream model \cite{bergeron1993rc,fomin1994grothendieck, KNUTSON2004161,knutson2005grobner, lenart2006grothendieck} and  the relatively newly developed bumpless pipe dream model \cite{LLS,LLS-2, Wei,BS}.

In this paper, we work in the setting of $\beta$-Grothendieck polynomials, focusing on   the principal specialization of  $\mathfrak{G}_w^{(\beta)}(x)$:
\[
\Upsilon_w{(\beta)}:=\mathfrak{G}_w^{(\beta)}(\mathbf{1})=\mathfrak{G}_w^{(\beta)}(x)|_{x_i=1}.
\]
Combinatorially, $\Upsilon_w{(\beta)}$ is a weighted counting of pipe dreams or bumpless pipe dreams of $w$.
Setting  $\beta=0$ gives  the principal specialization of the Schubert polynomial:
\[
\Upsilon_w:=\Upsilon_w(0)=\mathfrak{S}_w(\mathbf{1}).
\]
There has been different notation for principal specializations in the literature, and here we adopt the notation used in \cite{anderson2026computationsamplingschubertspecializations, MPP-2, morales2025grothendieck}. 
We design an algorithm on pipe dreams which maps injectively  a pipe dream of a permutation $w$ to a   pipe dream corresponding to a subword of $w$. When $w$ does not contain the $1423$ pattern, we prove that the map is a bijection. As a consequence, we deduce  that  for the class of permutations $w$ avoiding  $1423$, $\Upsilon_w{(\beta)}$ can be formulated nonnegatively in terms of  the occurrences of  patterns in $w$. We also find an inverse conservation principle which enables  us to obtain the same nonnegativity for  permutations avoiding  the $1342$ pattern. Our results partially resolve  conjectures due to Gao \cite{gao2021principal} (independently by Gaetz), Me\'sz\'aros and Tanjaya \cite{meszaros2022inclusion}, and Dennin \cite{ALCO_2025__8_3_745_0}, respectively.

\subsection{Background on principal specializations} A fundamental  formula for the specialization of Schubert polynomials was given by  Macdonald \cite{Mac}  (see  \cite{FS,HPSW} for alternative proofs),  stating   that 
\[
\Upsilon_w=\frac{1}{\ell !}\sum_{(a_1,\ldots, a_\ell)\in \mathrm{Red}(w)} a_1\cdots a_\ell,
\]
where $\ell$ is the length of $w$, and  the sum is over reduced words of $w$. In the  ``Schubert shenanigans'' paper \cite{stanley},  Stanley asked various enumeration and 
asymptotic/extremal problems concerning $\Upsilon_w$. 
These problems were systematically 
investigated, and for recent developments  see for example  Weigandt  \cite{weigandt2018schubert},  Morales,   Pak  and   Panova \cite{MPP-1}, Gao \cite{gao2021principal}, Me\'sz\'aros and Tanjaya \cite{meszaros2022inclusion}, Zhang \cite{Zhang-1}, Guo and Lin \cite{guo2024schubert}, Chou and  Setiabrata \cite{CS},  and the very near publication by  
 Anderson,   Panova  and   Petrov \cite{anderson2026computationsamplingschubertspecializations}. 

 Of particular interest among previous work is the possible  relationship  between   principal specializations and patterns in permutations. Let $S_n$ denote the set of permutations of $\{1,2,\ldots, n\}$. We allow $n=0$ and in this case  set $S_0=\{\emptyset\}$. We shall use the one-line notation, that is, write $w\in S_n$ as $w=w(1)\cdots w(n)$. 
 For two words $u=u(1)\cdots u(m)$ and $v=v(1)\cdots v(n)$ of positive integers, 
  write $u\leq v$ if $u$ appears as a subword in $v$, that is,  there exist $1\leq i_1<\cdots<i_m\leq n$ such that $u(k)=v(i_k)$ for $1\leq k\leq m$. We declare  $\emptyset\leq v$ for every subword $v$. Throughout, all words under consideration have distinct entries. Define $\mathrm{perm}(u)$ as the permutation whose entries have the same relative order as $u$. For example, $\operatorname{perm}(2574)=1342$. Given   permutations  $w\in S_n$ and $u\in S_m$ where $m\leq n$, we say that a subword $w(i_1)\cdots w(i_m)$    of $w$ is  a $u$ pattern if  $u=\mathrm{perm}(w(i_1)\cdots w(i_m))$. Let $p_u(w)$ denote the number of occurrences  of the  $u$ pattern  in $w$. For example, we have $p_{132}(1432)=3$. Note that we always set $p_{\emptyset}(w)=1$.

For $w\in S_n$, define $c_w$ recursively  by setting $c_\emptyset=1$ and for $n\geq 1$, 
\[
c_w := \Upsilon_w - \sum_{\text{$v\in S_m$ with $0\leq m<n$}} c_v\cdot p_v(w).
\]
Equivalently, this  yields  the following formulation  for the specialization $\Upsilon_w$:
\begin{equation}\label{enq-11}
\Upsilon_w=\sum_{v} c_v\cdot p_v(w).
\end{equation}
Gao \cite{gao2021principal}  made the following bold prediction, which, as remarked in \cite[Remark 4.3]{gao2021principal},  was also observed independently by Christian Gaetz.

\begin{conjecture}[\cite{gao2021principal}]\label{conj:Gao}
For any permutation $w$, we have
$c_{w}\geq 0$. Equivalently,
\[
\Upsilon_w=\sum_{v} c_v\cdot p_v(w), \ \ \ \text{where $c_v\in \mathbb{Z}_{\geq 0}$}.
\]
\end{conjecture}

Note that Conjecture \ref{conj:Gao} would immediately imply the lower bounds for $\Upsilon_w$ which have previously  been established in \cite{weigandt2018schubert, gao2021principal, guo2024schubert}. 

Me\'sz\'aros and Tanjaya  \cite{meszaros2022inclusion} noticed   that \eqref{enq-11} can be equivalently expressed  as 
 \[
\Upsilon_w=\sum_{\emptyset\leq v\leq w} c_{\mathrm{perm}(v)}.
 \]
Using inclusion-exclusion, it follows that
\[
c_w=\sum_{\emptyset\leq v\leq w} (-1)^{|w|-|v|} \Upsilon_{\mathrm{perm}(v)},
\]
where $|\cdot|$ stands for the number of entries in a word.
Me\'sz\'aros and Tanjaya \cite{meszaros2022inclusion}
proposed  the following stronger conjecture. 

\begin{conjecture}[\cite{meszaros2022inclusion}]\label{conj:mes} 
For any permutation $w$ and  a subword $u$ of $w$, 
\[
 \sum_{u\le v\le w} (-1)^{|w|-|v|}\Upsilon_{\mathrm{perm}(v)}\ge0.
\]
This reduces to Conjecture \ref{conj:Gao} in the case when $u=\emptyset$. 
\end{conjecture}

M\'esz\'aros and Tanjaya \cite{meszaros2022inclusion} proved Conjecture \ref{conj:mes}  in the case when $w$ avoids both the   $1423$  pattern  and the  $1432$ pattern. Their proof relies on the realization of Schubert polynomials as the characters of flagged Weyl modules, which  have proved  to be  a powerful tool in the recent study of   Schubert polynomials \cite{2018Schubert,fink2021zero, guo2024schubert}.

The story surrounding principal specializations of Grothendieck polynomials is relatively new. 
When $w$ is a vexillary permutation (that is, a permutation avoiding the $2143$ pattern), the specialization $\Upsilon_w{(\beta)}$ has been considered by   Morales,  Pak and Panova \cite{MPP-2}. For recent progress, we refer to  the ``Grothendieck shenanigans'' paper by   Morales,   Panova,  Petrov and Yeliussizov \cite{morales2025grothendieck}.  
We shall next describe a conjecture by   Dennin \cite{ALCO_2025__8_3_745_0}, which can be thought of as a Grothendieck analogue  of Conjecture \ref{conj:Gao}.  
For $w\in S_n$, define  $c_w{(\beta)} \in \mathbb{Z}[\beta]$ recursively as follows: set $c_{\emptyset}{(\beta)}= 1$, and for $n\geq 1$, let 
\[
    c_w{(\beta)} := \Upsilon_w{(\beta)}- \sum_{\text{$v\in S_m$ with $0\leq m<n$}} c_v{(\beta)}\cdot p_v(w).
    \]
    Equivalently,
\[
c_w{(\beta)}=\sum_{\emptyset\leq v\leq w} (-1)^{|w|-|v|} \Upsilon_{\mathrm{perm}(v)}{(\beta)}.
\]

 Dennin \cite{ALCO_2025__8_3_745_0}
conjectured that $c_w{(\beta)} $ has nonnegative coefficients. 

\begin{conjecture}[\cite{ALCO_2025__8_3_745_0}]\label{Conj:Hugn}
    For any permutation \( w\), we have \( c_w{(\beta)} \in \mathbb{Z}_{\geq 0}[\beta] \). Equivalently,\[
\Upsilon_w(\beta)=\sum_{v} c_v(\beta)\cdot p_v(w), \ \ \ \text{where $c_v(\beta)\in \mathbb{Z}_{\geq 0}[\beta]$}.
\]
\end{conjecture}

Letting $\beta=0$,   Conjecture \ref{Conj:Hugn} specializes  to  Conjecture \ref{conj:Gao}. 
Dennin \cite{ALCO_2025__8_3_745_0} settled Conjectu re \ref{Conj:Hugn} for permutations avoiding simultaneously  the two   patterns $1243$ and $2143$. 
The proof used in \cite{ALCO_2025__8_3_745_0} relies on   the bumpless pipe dream model of Grothendieck polynomials. 
In the Schubert setting, the same technique   can prove Conjecture \ref{conj:Gao} for permutations avoiding the single $1243$ pattern  \cite{ALCO_2025__8_3_745_0}.

\subsection{Our results} In this paper, 
we prove Conjecture \ref{Conj:Hugn} for two new classes of permutations, each characterized by the avoidance of a single pattern of length four.

\begin{theorem}\label{thm:Grothendieck_main}
Conjecture \ref{Conj:Hugn} holds for every permutation $w$ avoiding the $1423$ pattern.  
\end{theorem}

The prove   Theorem \ref{thm:Grothendieck_main}, we employ   the pipe dream model for Grothendieck polynomials. We construct  an algorithm which maps  a pipe dream of a permutation $w$ injectively  to a smaller pipe dream corresponding to a subword of $w$. When $w$ avoids the $1423$ pattern, we prove that the injection is in fact  a bijection.  
This would allow us to interpret    the polynomial $c_w(\beta)$  as a weight counting of certain  pipe dreams of $w$. The techniques developed in  \cite{meszaros2022inclusion} or \cite{ALCO_2025__8_3_745_0} do  not seem to apply directly to       Theorem \ref{thm:Grothendieck_main}.  Compared with the operations on bumpless pipe dreams  used in \cite{ALCO_2025__8_3_745_0}, our  algorithm  on pipe dreams is essentially different. 

In fact, we shall prove a stronger result.

\begin{theorem}\label{main-1}
Fix a permutation $w$ and a subword $u$  of $w$. If $w$ avoids the $1423$ pattern, then 
\[
\sum_{u\leq v\leq w} (-1)^{|w|-|v|} \Upsilon_{\mathrm{perm}(v)}{(\beta)} \in \mathbb{Z}_{\geq 0}[\beta].
\]
This yields  Theorem \ref{thm:Grothendieck_main} by letting $u=\emptyset$. 
\end{theorem}

Taking $\beta=0$ in Theorem \ref{main-1}, we arrive at the following.

\begin{corollary}\label{thm:Schubert_main}
 Conjecture \ref{conj:mes} (and thus Conjecture \ref{conj:Gao}) is true for permutations avoiding the $1423$ pattern.
\end{corollary}

\begin{remark}
Recall that M\'esz\'aros and Tanjaya \cite{meszaros2022inclusion} proved Conjecture\ref{conj:mes} for permutations simultaneously  avoids the     $1423$   and    $1432$ patterns.  Our approach  leads to a new proof of their  result.  We will further discuss this in Section  \ref{last-sec}.
\end{remark}

In view of Conjecture \ref{conj:mes}
and Theorem \ref{main-1}, we are inspired  to make the following conjecture.  

\begin{conjecture}\label{1-7}
 For a permutation $w \in S_n$ and  a subword $u$ of $w$, we have
\[
 \sum_{u\le v\le w} (-1)^{|w|-|v|}\Upsilon_{\mathrm{perm}(v)}(\beta) \in \mathbb{Z}_{\geq 0}[\beta].
\]   
\end{conjecture}

We now have the following implications. 
 
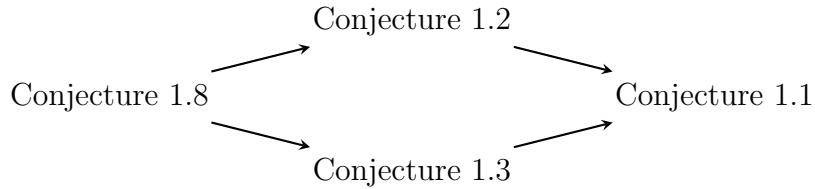
\[
    \begin{tikzpicture}[
    >=stealth, 
    thick,     
    node font=\normalsize 
]

    \node (c17) at (0, 0)   {Conjecture \ref{1-7}};
    \node (c13) at (4, 1)   {Conjecture \ref{conj:mes} };
    \node (c12) at (4, -1)  {Conjecture \ref{Conj:Hugn}};
    \node (c11) at (8, 0)   {Conjecture  \ref{conj:Gao}};

    \draw[->] (c17) -- (c13);
    \draw[->] (c17) -- (c12);
    \draw[->] (c13) -- (c11);
    \draw[->] (c12) -- (c11);
\end{tikzpicture}
 \]

Interestingly, the constructions in our  proof lead   us to be aware of  the following phenomenon, which is easy to justify, but does not seem to appear in the literature. 

\vspace{10pt}
\noindent
\textbf{Inverse conservation principle.}
For any permutation  $w$, we have   \( c_w{(\beta)}= c_{w^{-1}}{(\beta)}\).

\begin{proof}
Suppose that $w\in S_n$. Let us use induction on $n$. The assertion  is clear when $n=0$. Now consider $n\geq 1$. Notice that 
\begin{align}
      c_{w^{-1}}{(\beta)}&= \Upsilon_{w^{-1}}{(\beta)}- \sum_{\text{$v\in S_m$ with $0\leq m<n$}} c_v{(\beta)}\cdot p_v({w^{-1}})\nonumber\\
   &= \Upsilon_{w^{-1}}{(\beta)}- \sum_{\text{$v\in S_m$ with $0\leq m<n$}} c_{v^{-1}}{(\beta)}\cdot p_{v^{-1}}({w^{-1}})\label{agj-0}
\end{align}
Proposition \ref{icin} tells that $\Upsilon_w{(\beta)}=\Upsilon_{w^{-1}}{(\beta)}$. Note also that $p_v(w)=p_{v^{-1}}({w^{-1}})$. Using induction, \eqref{agj-0} becomes 
\begin{align*}
      c_{w^{-1}}{(\beta)}
      =  \Upsilon_w{(\beta)}- \sum_{\text{$v\in S_m$ with $0\leq m<n$}} c_{v}{(\beta)}\cdot p_{v}({w})= c_{w}{(\beta)},
\end{align*}
as desired. 
\end{proof}

Combining  Theorem \ref{thm:Grothendieck_main}  with the   inverse conservation principle, we obtain a new class of permutations supporting Conjecture \ref{Conj:Hugn}.

\begin{corollary}\label{cor:Grothendieck_main}
 Conjecture \ref{Conj:Hugn} is true for any permutation $w$ avoiding the $1342$ pattern. 
\end{corollary}

\begin{proof}
 Notice that the inverse  of $1423$ is $1342$. 
\end{proof}

This paper is arranged   as follows. In Section \ref{groth-99}, we review the pipe dream model for Grothendieck polynomials.   Section \ref{amin-p-0} is devoted to the proofs of our main theorems.
Some further remarks are given  in  Section  \ref{last-sec}.

\subsection*{Acknowledgements}
This work was  supported by the National Natural Science Foundation of China (No. 12371329) and the Fundamental Research Funds for the Central Universities (No. 63243072).

\section{Pipe dreams for Grothendieck polynomials}\label{groth-99}

In this section, we give an overview of the  pipe dream construction of  $\beta$-Grothendieck polynomials, see for example   \cite{fomin1994grothendieck, KNUTSON2004161,knutson2005grobner, lenart2006grothendieck} for further information. We shall  adopt the description by Lenart, Robinson and   Sottile   \cite{ lenart2006grothendieck} which avoids interpreting the $0$-Hecke algebra product. 

\subsection{Divided difference operators}

For a permutation $w\in S_n$ and a polynomial $f\in  \mathbb{Z}[\beta][\mathbf{x}]=\mathbb{Z}[\beta][x_1,\ldots, x_n]$, $w$ acts on $f$ by 
\[
w\cdot f=f(x_{w(1)},\ldots, x_{w(n)}). 
\]
Define the  divided difference operator  $\partial_i$ acting on $\mathbb{Z}[\beta][\mathbf{x}]$ by
\[
\partial_i(f)=\frac{f- s_i\cdot f}{x_i-x_{i+1}},
\]
where $s_i$ is the simple transposition swapping $i$ and $i+1$. 

Recall that the length $\ell(w)$ of a permutation $w$ is the minimum number of simple transpositions that can be used to generate $w$. It admits a combinatorial description 
\[
\ell(w)=|\{(i,j)\colon 1\leq i<j\leq n,\, w(i)>w(j)\}|.
\]
Let $w_0=n\cdots 2 1$ be the  permutation in $S_n$ with longest length. Set \[\mathfrak{G}_{w_0}^{(\beta)}(x)=x_1^{n-1}x_2^{n-2}\cdots x_{n-1}.\]
If $w\neq w_0$, then there must exist an index $1\leq i<n$ such that $w(i)<w(i+1)$. Let $ws_i$ be the permutation obtained from $w$ by swapping $w(i)$ and $w(i+1)$. Note that $\ell(ws_i)=\ell(w)+1$.    Define 
\[
\mathfrak{G}_{w}^{(\beta)}(x)=\pi_i \,\mathfrak{G}_{ws_i}^{(\beta)}(x), 
\]
where $\pi_i=\partial_i\,(1+\beta x_{i+1})$. 
The above definition is independent of the choice of $i$ since the operators $\pi_i$ satisfy the Coxeter  relations
\[
\pi_i\,\pi_j=\pi_j\,\pi_i\ \ \ \text{for $|i-j|>1$}, \ \ \ \ \pi_i\pi_{i+1}\pi_i=\pi_{i+1}\pi_i\pi_{i+1}. 
\]

\subsection{Pipe dreams}

A pipe dream of rank $n$ is a tiling with $n+1-i$ left-justified tiles in
row $i$ by using two types of tiles: the bumping tile $\BPD{\B}$ and the crossing tile $\BPD{\X}$, with the constraint that each tile lying on the southeast  diagonal is $\BPD{\B}$. Intuitively, each pipe dream has $n$ pipes, which are  labeled $1,2,\ldots, n$ from left to right, and these pipes  enter   from the top boundary, going down and to the left,  and eventually  exit from the left boundary. We shall use $(i,j)$ to denote the tile in row $i$ and column $j$. As usual, we use matrix coordinates to index rows and columns, that is,    rows are counted from top to bottom, and columns are counted from left to right.

A pipe dream is  reduced if  any two pipes can cross at most once. 
For a reduced pipe dream, the associated permutation can be obtained by  reading the labels of the pipes along the left boundary from top to bottom.  
 Let $\operatorname{RPD}(w)$ denote the set of all reduced pipe dreams of $w$. For example, in Figure \ref{fig:nonreduced_pipe_dreams_of_2143}, the first three pipe  dreams on the top row  constitute all reduced pipe dreams of $2143$. As will be seen later, the set of reduced pipe dreams of $w$ provides a combinatorial model for the Schubert polynomial $\mathfrak{S}_w(x)$. 
 
 Concerning  Grothendieck polynomials, we need to consider general pipe dreams \cite{fomin1994grothendieck, KNUTSON2004161,knutson2005grobner, lenart2006grothendieck}. Here we adopt the description in \cite{ lenart2006grothendieck} which avoids using  the $0$-Hecke algebra language to the maximum extent.  
 Let $P\in \mathrm{RPD}(w)$ be a reduced pipe dream of $w$. For each bumping tile $\BPD{\B}$, say   $(i,j)$, in $P$, one may possibly  mark the pipe pieces  within  it with  colors if the two pipes traversing it cross
at some position  above row $i$. Such a marked bumping tile will be denoted   \(\BPD{\MB}\). The reduced pipe dream $P$, together with some marked bumping tiles, is called a marked reduced pipe dream of $w$. Let $\mathrm{MRPD}(w)$  stand for the set of all marked reduced pipe dreams of $P$. 
 For example,  Figure \ref{fig:nonreduced_pipe_dreams_of_2143} lists all marked reduced pipe dreams of $2143$. 
\begin{figure}[h t]
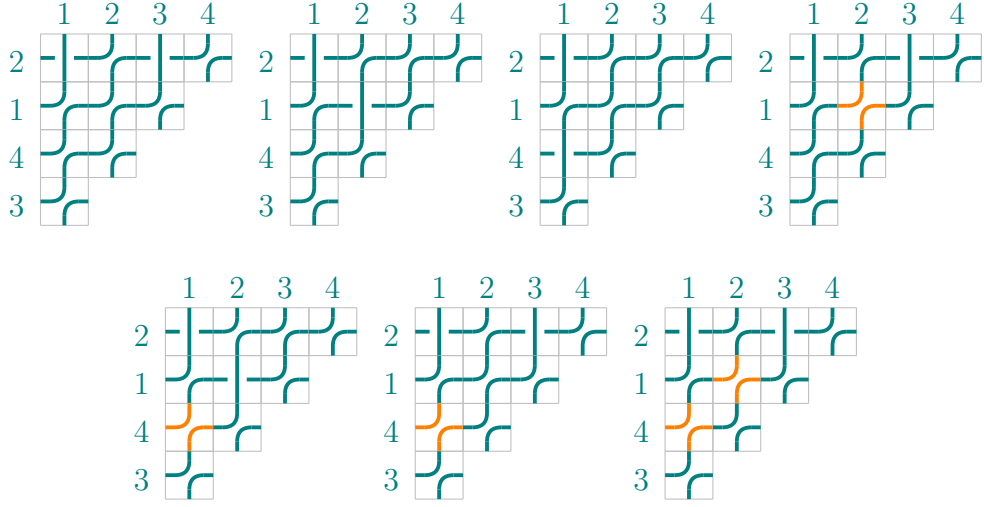

    \centering
\[\BPD[1.5pc]{
\M{}\M{1}\M{2}\M{3}\M{4}\\   
\M{2}\X\B\X\B \\
\M{1}\B\B\B \\
\M{4}\B\B \\
\M{3}\B \\
}\BPD[1.5pc]{
\M{}\M{1}\M{2}\M{3}\M{4}\\   
\M{2}\X\B\B\B \\
\M{1}\B\X\B \\
\M{4}\B\B \\
\M{3}\B \\
}\BPD[1.5pc]{
\M{}\M{1}\M{2}\M{3}\M{4}\\   
\M{2}\X\B\B\B \\
\M{1}\B\B\B \\
\M{4}\X\B \\
\M{3}\B \\
}\BPD[1.5pc]{
\M{}\M{1}\M{2}\M{3}\M{4}\\   
\M{2}\X\B\X\B \\
\M{1}\B\MB\B \\
\M{4}\B\B \\
\M{3}\B \\
}\]
\[\BPD[1.5pc]{
\M{}\M{1}\M{2}\M{3}\M{4}\\   
\M{2}\X\B\B\B \\
\M{1}\B\X\B \\
\M{4}\MB\B \\
\M{3}\B \\
}\BPD[1.5pc]{
\M{}\M{1}\M{2}\M{3}\M{4}\\   
\M{2}\X\B\X\B \\
\M{1}\B\B\B \\
\M{4}\MB\B \\
\M{3}\B \\
}\BPD[1.5pc]{
\M{}\M{1}\M{2}\M{3}\M{4}\\   
\M{2}\X\B\X\B \\
\M{1}\B\MB\B \\
\M{4}\MB\B \\
\M{3}\B \\
}
\]
\caption{All   marked reduced pipe dreams of $2143$, among which the first three on the top row are the ordinary reduced pipe dreams. }
\label{fig:nonreduced_pipe_dreams_of_2143}
\end{figure}

For  $P\in \mathrm{MRPD}(w)$, denote
\[
x^P= \prod_{(i,j)} x_i,
\]
where $(i,j)$  ranges over all   crossing tiles as well as marked bumping tiles  in $P$. The following expression  for $\beta$-Grothendieck polynomials was due to Fomin and Kirillov \cite{fomin1994grothendieck}, which has been further investigated for example in   \cite{KNUTSON2004161,knutson2005grobner, lenart2006grothendieck} from the perspectives of geometry and commutative algebra.

\begin{theorem}\label{thm:beta_Gro}
For any permutation $w\in S_n$,
\[
    \mathfrak{G}_w^{(\beta)}(x) = \sum_{P\in \operatorname{MRPD}(w)} \beta^{   \mathrm{mbt}(P)}x^{P},
\]
where $\mathrm{mbt}(P) $ is the  number of marked bumping tiles of $P$. 
\end{theorem}

When $\beta=0$, the contribution from  pipe dreams with marked bumping tiles vanishes, leading to the reduced pipe dream  formula for Schubert polynomials:
\[
    \mathfrak{S}_w(x) = \sum_{P\in \operatorname{RPD}(w)} x^{P}.
\]

\begin{corollary}
  For any permutation $w\in S_n$,
  \[
 \Upsilon_w{(\beta)}= \sum_{P\in \operatorname{MRPD}(w)} \beta^{ \mathrm{mbt}(P)}.
  \]
\end{corollary}

Consider  \(w=2143\). From Figure \ref{fig:nonreduced_pipe_dreams_of_2143}, it follows that   
\[
\mathfrak{G}_w^{(\beta)}(x) =\left(x_{1}^{2}+x_{1}x_{2}+x_{1}x_{3}\right)+\beta\left(x_{1}^{2}x_{2}+x_{1}^{2}x_{3}+x_{1}x_{2}x_{3} \right)+\beta^{2}\left(x_{1}^{2}x_{2}x_{3}\right),\]
and 
\[
\Upsilon_w{(\beta)}=3+3\beta+\beta^2.
\]

The following relationship is needed  in the proof of the ``inverse conservation principle'', as already exhibited  in the introduction. 

\begin{proposition}\label{icin}
  For any permutation $w\in S_n$,
  \[
 \Upsilon_w{(\beta)}= \Upsilon_{w^{-1}}{(\beta)}. 
 \]
\end{proposition}

\begin{proof}
There is a dual version of marked reduced   pipe dreams, see \cite[Section 6]{lenart2006grothendieck} where they are called   left marked RC-graphs.  That is, a bumping tile may be marked if the two pipes   have already crossed at some position southwest of it. 
So,  transposing pipe dreams along the main diagonal gives a bijection between the marked reduced pipe dreams of $w$ and $w^{-1}$. This completes  the proof.    
\end{proof}

\section{Proofs of   Theorems \ref{thm:Grothendieck_main} and  \ref{main-1}}\label{amin-p-0}

In this section, we shall give proofs of Theorems \ref{thm:Grothendieck_main} and  \ref{main-1}. 
The starting point is to establish a reduction   algorithm, transforming  a marked reduced pipe dream of a permutation injectively into a smaller marked reduced pipe dream of its subword. To do this, we first distinguish  pipes that can be erased from a marked  pipe dream. 

\subsection{Removable pipes and core pipe dreams}

Let $w\in S_n$ and  $P$ be a marked reduced  pipe dream of   $w$.
We say that a pipe in $P$ with label $j$   is removable if the following conditions are satisfied 
\begin{enumerate}
    \item [(i)]  In the column where pipe $j$ enters, all \(\BPD{\X}\)'s    are traversed by  pipe $j$, and there is no marked bumping tile \(\BPD{\MB}\). In other words, this column  is formed (from top to bottom) by some \(\BPD{\X}\)'s,  followed by some $\BPD{\B} $'s.

    \item [(ii)] For each crossing tile \(\BPD{\X}\) traversed by pipe $j$, the tile  directly above it  (if any) cannot be an unmarked  bumping tile  $\BPD{\B} $.

    \item [(iii)] 
    There is no marked bumping tile \(\BPD{\MB}\) traversed by pipe $j$. 
\end{enumerate}
Otherwise, we say that pipe $j$ is non-removable.
In Figure \ref{fig:Two_non_reduced_pipe_dreams_of_132564}, we illustrate two marked reduced pipe dreams of \(w=132564\). The labels of all removable pipes are encircled by  colored boxes.
\begin{figure}[h]
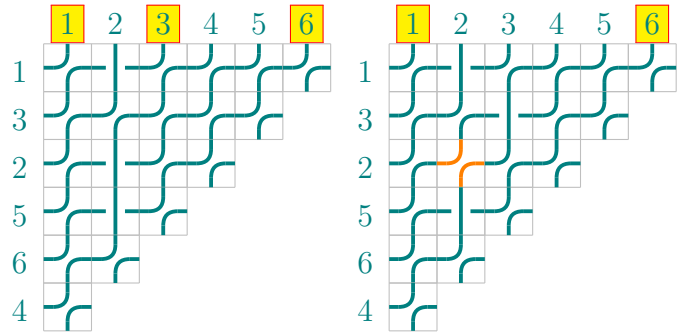

    \centering

    \[
\BPD[1.5pc]{
\M{}\M{\fcolorbox{red}{yellow}{1}}\M{2}\M{\fcolorbox{red}{yellow}{3}}\M{4}\M{5}\M{\fcolorbox{red}{yellow}{6}}\\   
\M{1}\B\X\B\B\B\B \\
\M{3}\B\B\B\B\B \\
\M{2}\B\X\B\B \\
\M{5}\B\X\B \\
\M{6}\B\B \\
\M{4}\B \\
}
\BPD[1.5pc]{
\M{}\M{\fcolorbox{red}{yellow}{1}}\M{2}\M{3}\M{4}\M{5}\M{\fcolorbox{red}{yellow}{6}}\\   
\M{1}\B\X\B\B\B\B \\
\M{3}\B\B\X\B\B \\
\M{2}\B\MB\B\B \\
\M{5}\B\X\B \\
\M{6}\B\B \\
\M{4}\B \\
}
\]
    
    \caption{Illustration of removable pipes.}
    \label{fig:Two_non_reduced_pipe_dreams_of_132564}
\end{figure}

A marked reduced  pipe dream   is called a core  if none of its pipes  is removable.
The name  is inspired by the notion of a core partition in representation theory \cite{JK} (recall that a partition is a $t$-core if one  cannot remove any $t$-rim hook from it). 
Let $\mathrm{CMRPD}(w)$ be the set of core  marked reduced pipe dreams of $w$, and $\mathrm{CRPD}(w)$ be the subset  consisting of those  without marked bumping tiles. For example, as illustrated in Figure \ref{fig:nonreduced_pipe_dreams_of_2143-23}, the permutation $2143$ has two core marked reduced  pipe dreams, both of which have marked bumping tiles.
\begin{figure}[h t]
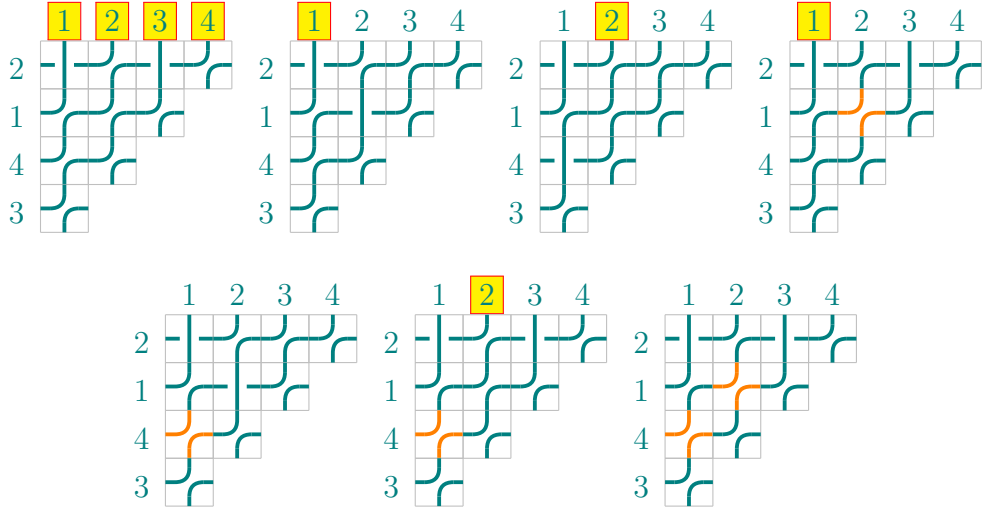

    \centering
\[\BPD[1.5pc]{
\M{}\M{\fcolorbox{red}{yellow}{1}}\M{\fcolorbox{red}{yellow}{2}}\M{\fcolorbox{red}{yellow}{3}}\M{\fcolorbox{red}{yellow}{4}}\\   
\M{2}\X\B\X\B \\
\M{1}\B\B\B \\
\M{4}\B\B \\
\M{3}\B \\
}\BPD[1.5pc]{
\M{}\M{\fcolorbox{red}{yellow}{1}}\M{2}\M{3}\M{4}\\   
\M{2}\X\B\B\B \\
\M{1}\B\X\B \\
\M{4}\B\B \\
\M{3}\B \\
}\BPD[1.5pc]{
\M{}\M{1}\M{\fcolorbox{red}{yellow}{2}}\M{3}\M{4}\\   
\M{2}\X\B\B\B \\
\M{1}\B\B\B \\
\M{4}\X\B \\
\M{3}\B \\
}\BPD[1.5pc]{
\M{}\M{\fcolorbox{red}{yellow}{1}}\M{2}\M{3}\M{4}\\   
\M{2}\X\B\X\B \\
\M{1}\B\MB\B \\
\M{4}\B\B \\
\M{3}\B \\
}\]
\[\BPD[1.5pc]{
\M{}\M{1}\M{2}\M{3}\M{4}\\   
\M{2}\X\B\B\B \\
\M{1}\B\X\B \\
\M{4}\MB\B \\
\M{3}\B \\
}\BPD[1.5pc]{
\M{}\M{1}\M{\fcolorbox{red}{yellow}{2}}\M{3}\M{4}\\   
\M{2}\X\B\X\B \\
\M{1}\B\B\B \\
\M{4}\MB\B \\
\M{3}\B \\
}\BPD[1.5pc]{
\M{}\M{1}\M{2}\M{3}\M{4}\\   
\M{2}\X\B\X\B \\
\M{1}\B\MB\B \\
\M{4}\MB\B \\
\M{3}\B \\
}
\]
\caption{All marked reduced pipe dreams of $2143$. Furthermore, the first and third  on the bottom row are core pipe dreams. }
\label{fig:nonreduced_pipe_dreams_of_2143-23}
\end{figure}

\subsection{The reduction  algorithm}

 Let $w\in S_n$ and $P\in \mathrm{MRPD}(w)$ be a  marked reduced pipe dream of $w$. We shall introduce  an operation mapping   $P$ to a marked reduced pipe dream $P'$ of a subword $v$ of $w$. Here a  marked reduced pipe dream of the subword $v$ is defined in the same way as for a permutation, except by replacing the pipe labels by the entries of $v$.  Hence if we read the labels on the left boundary of $P'$ from top to bottom,  then we obtain the subword $v$. For example, Figure \ref{fig:Pipe_Dream_of_subword} depicts  a  marked reduced pipe dream of the subword $v=2563$. 
 \begin{figure}[h t]
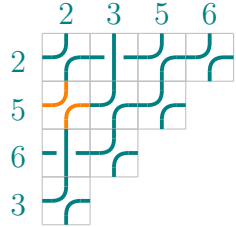

     \centering
     \[\BPD[1.5pc]{
    \M{}\M{2}\M{3}\M{5}\M{6}\\   
    \M{2}\B\X\B\B \\
    \M{5}\MB\B\B \\
    \M{6}\X\B \\
    \M{3}\B \\
    }\]
     \caption{A  marked reduced pipe dream of the subword $v=2563$.}
     \label{fig:Pipe_Dream_of_subword}
 \end{figure}

Now fix  $P\in \mathrm{MRPD}(w)$, and assume that  pipe $j$ is a removable pipe of $P$. Let $v$ be the subword of $w$ by deleting the entry $j$. We shall define a map
\[
\Phi_{j}\colon \mathrm{MRPD}(w)\longrightarrow\mathrm{MRPD}(v)
\]by performing two  operations: deletion  and merging.

\begin{itemize}
    \item[]  {\bf Deletion.} 
Let $C_j$ denote the column of $P$ where pipe $j$ enters.     This operation will be performed on the columns including $C_j$ as well as those to its left, which are exactly the columns traversed by pipe $j$. 
    \begin{itemize}
        \item[(1)] Delete the whole column $C_j$.

        \item[(2)] For any fixed  column strictly to the left of $C_j$, consider the tiles   in this column  traversed by pipe $j$.  By the conditions (ii) and (iii) in the definition of a removable pipe, it is readily checked that the tiles traversed by pipe $j$ are 
        \begin{itemize}
            \item[$\bullet$]  either a single crossing tile $\BPD{\X}$ where pipe $j$ traverses horizontally,
            
            \item[$\bullet$] or two consecutive  bumping tiles $\BPD{\delBDown\\ \delBUP\\}$ where the pipe piece  between the two red diagonals belongs to pipe $j$. 
        \end{itemize}
         In the first case, delete the tile $\BPD{\X}$. In the latter case, delete the two `semi-tiles' $\BPD{\delBDown\\ \delBUP\\}$ lying between the two red  diagonal  lines. 
    \end{itemize}
For an illustration of the deletion procedure, see the first two steps in Figure \ref{fig:the_operation_of_Phi_{w,i,u}},  where the deleted areas are encircled by red lines. 
\end{itemize}

\begin{itemize}
    \item[]  {\bf Merging.}  After performing this operation, we will get a new pipe diagram whose rank is decreased by $1$. To be specific,
    
    \begin{itemize}
        \item[(1)]  Shift all the tiles, which lie strictly to the right of   $C_j$, leftward  by one unit. 

        \item[(2)] Shift all the tiles as well as some possible `semi-tiles' in the columns to the left of $C_j$, which lie below pipe $j$, upward by one unit. 
    \end{itemize}
See the last step in Figure \ref{fig:the_operation_of_Phi_{w,i,u}} for an illustration. 
\end{itemize}

\vspace{-10pt}
\begin{figure}[h t]
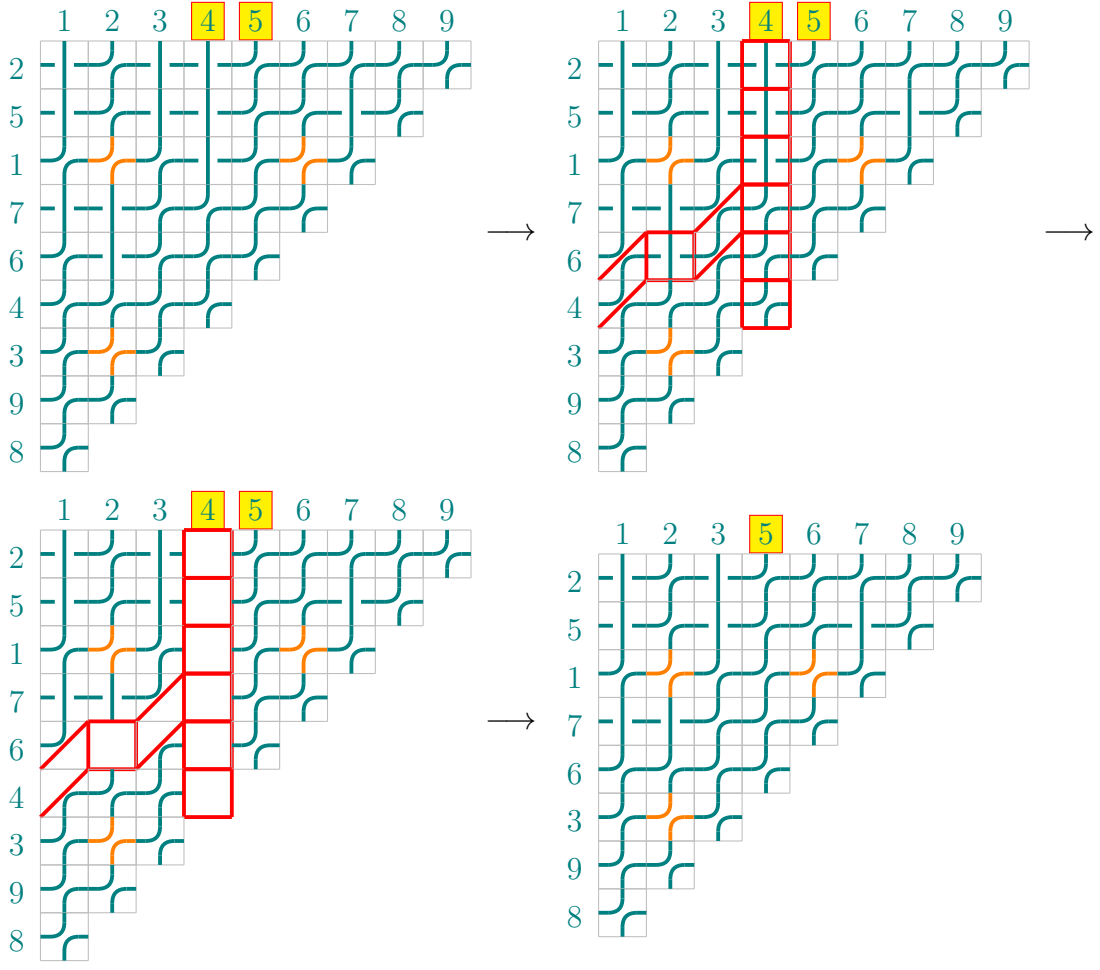

    \centering
    \[
    \begin{aligned}
    &\BPD[1.5pc]{
        \M{}\M{1}\M{2}\M{3}\M{\fcolorbox{red}{yellow}{4}}\M{\fcolorbox{red}{yellow}{5}}\M{6}\M{7}\M{8}\M{9}\\   
        \M{2}\X\B\X\X\B\B\B\B\B \\
        \M{5}\X\B\X\X\B\B\X\B \\
        \M{1}\B\MB\B\X\B\MB\B \\
        \M{7}\X\X\B\B\B\B \\
        \M{6}\B\X\B\B\B \\
        \M{4}\B\B\B\B \\
        \M{3}\B\MB\B \\
        \M{9}\B\B\\
        \M{8}\B\\
    }
    \longrightarrow
    \BPD[1.5pc]{
        \M{}\M{1}\M{2}\M{3}\M{\fcolorbox{red}{yellow}{4}}\M{\fcolorbox{red}{yellow}{5}}\M{6}\M{7}\M{8}\M{9}\\   
        \M{2}\X\B\X\delX\B\B\B\B\B \\
        \M{5}\X\B\X\delX\B\B\X\B \\
        \M{1}\B\MB\B\delX\B\MB\B \\
        \M{7}\X\X\delBDown\delB\B\B \\
        \M{6}\delBDown\delX\delBUP\delB\B \\
        \M{4}\delBUP\B\B\delB \\
        \M{3}\B\MB\B \\
        \M{9}\B\B\\
        \M{8}\B\\
    }
    \longrightarrow \\
    &\BPD[1.5pc]{
        \M{}\M{1}\M{2}\M{3}\M{\fcolorbox{red}{yellow}{4}}\M{\fcolorbox{red}{yellow}{5}}\M{6}\M{7}\M{8}\M{9}\\   
        \M{2}\X\B\X\delO\B\B\B\B\B \\
        \M{5}\X\B\X\delO\B\B\X\B \\
        \M{1}\B\MB\B\delO\B\MB\B \\
        \M{7}\X\X\delJJ\delO\B\B \\
        \M{6}\delJJ\delO\delFF\delO\B \\
        \M{4}\delFF\B\B\delO\\
        \M{3}\B\MB\B \\
        \M{9}\B\B\\
        \M{8}\B\\
    }
    \longrightarrow
    \BPD[1.5pc]{
        \M{}\M{1}\M{2}\M{3}\M{\fcolorbox{red}{yellow}{5}}\M{6}\M{7}\M{8}\M{9}\\   
        \M{2}\X\B\X\B\B\B\B\B  \\
        \M{5}\X\B\X\B\B\X\B  \\
        \M{1}\B\MB\B\B\MB\B \\
        \M{7}\X\X\B\B\B  \\
        \M{6}\B\B\B\B \\
        \M{3}\B\MB\B \\
        \M{9}\B\B\\
        \M{8}\B\\
    }
    \end{aligned}
    \]

    \caption{Illustration of the map $\Phi_j$.}
    \label{fig:the_operation_of_Phi_{w,i,u}}
\end{figure}

The  resulting pipe diagram after the deletion/merging algorithm is defined as $\Phi_j(P)$. The following  proposition shows   that $\Phi_j$ is indeed a map from $\mathrm{MRPD}(w)$ to $\mathrm{MRPD}(v)$.

\begin{proposition}\label{lem:well_defined}
Assume that  $P\in \mathrm{MRPD}(w)$, and pipe $j$ is a removable pipe of $P$. Let $v$ be the subword of $w$ by deleting the entry $j$. Then $\Phi_j(P)$ is a marked reduced pipe dream of the subword $v$.
\end{proposition}

\begin{proof}
Write  $P'=\Phi_j(P)$.  By the construction of $\Phi_j$, $P'$ has rank one less than that of $P$. 
We shall  show that the route of any pipe $j'$ ($j'\neq j$) is not changed after applying  $\Phi_j$. Precisely, we justify   that if a  pipe piece in any undeleted tile (or semi-tile)  belongs to pipe $j'$ in $P$, then it still belongs to pipe $j'$ in $P'$. 

We still use $C_j$ to   denote the column of $P$ where pipe $j$ enters. 
Clearly, in each   column to the left of $C_j$, the merging operation does not change the connection of undeleted  pipes. 
It remains to  consider the effect of  the deletion/merging operation on  the tiles in   $C_j$. 
We need to analyze the local configuration of tiles in column
$C_j$, as well as the two columns immediately to the left and right of $C_j$.  All possible situations are listed in Figure \ref{vbisv-09}, where the top (resp., bottom) three rows indicate the configuration before (resp., after) the  deletion/merging operation. 
\begin{figure}[h t]
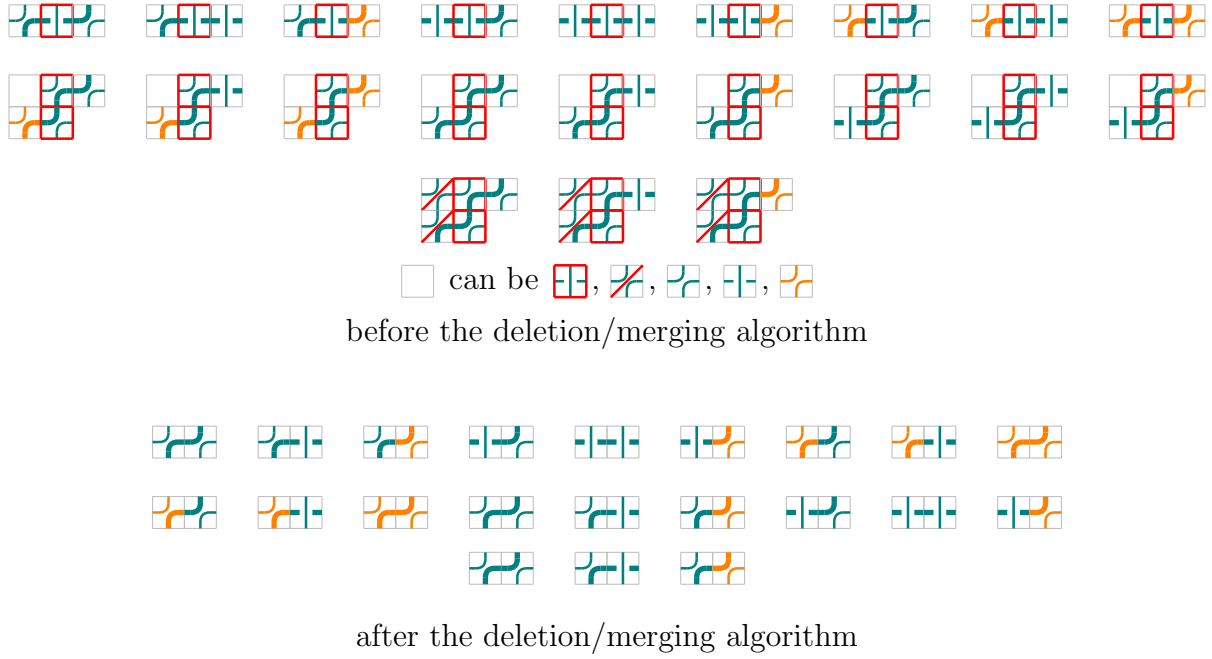

    \centering
    
               \[\BPD{\BoldBDown\BolddelX\BoldBUP} \quad \BPD{\BoldBDown\BolddelX\BoldX}\quad\BPD{\BoldBDown\BolddelX\BoldMBup} \quad \BPD{\BoldX\BolddelX\BoldBUP} \quad \BPD{\BoldX\BolddelX\BoldX}\quad \BPD{\BoldX\BolddelX\BoldMBup}\quad \BPD{\BoldMBdown\BolddelX\BoldBUP} \quad \BPD{\BoldMBdown\BolddelX\BoldX}\quad \BPD{\BoldMBdown\BolddelX\BoldMBup}\]\[\BPD{\O\BolddowndelB\BoldBUP\\\BoldMBdown\BoldupdelB} \quad \BPD{\O\BolddowndelB\BoldX\\\BoldMBdown\BoldupdelB}\quad \BPD{\O\BolddowndelB\BoldMBup\\\BoldMBdown\BoldupdelB}\quad\BPD{\O\BolddowndelB\BoldBUP\\\BoldBDown\BoldupdelB} \quad \BPD{\O\BolddowndelB\BoldX\\\BoldBDown\BoldupdelB}\quad \BPD{\O\BolddowndelB\BoldMBup\\\BoldBDown\BoldupdelB}\quad\BPD{\O\BolddowndelB\BoldBUP\\\BoldX\BoldupdelB} \quad \BPD{\O\BolddowndelB\BoldX\\\BoldX\BoldupdelB}\quad \BPD{\O\BolddowndelB\BoldMBup\\\BoldX\BoldupdelB}\]             \vspace{2pt}  
               \[\BPD{\delBDown\BolddowndelB\BoldBUP\\\BolddelBDown\BoldupdelB} \quad \BPD{\delBDown\BolddowndelB\BoldX\\\BolddelBDown\BoldupdelB}\quad\BPD{\delBDown\BolddowndelB\BoldMBup\\\BolddelBDown\BoldupdelB} \]
 
$\BPD{\O}$ can be $\BPD{\delX},\BPD{\delBUP},\BPD{\B},\BPD{\X},\BPD{\MB}$ 

               \vspace{5pt}
        before the deletion/merging algorithm

        \[\BPD{\BoldBDown \BoldBUP} \quad \BPD{\BoldBDown \BoldX}\quad\BPD{\BoldBDown \BoldMBup} \quad \BPD{\BoldX \BoldBUP} \quad \BPD{\BoldX \BoldX}\quad \BPD{\BoldX \BoldMBup}\quad \BPD{\BoldMBdown\BoldBUP} \quad \BPD{\BoldMBdown\BoldX}\quad \BPD{\BoldMBdown\BoldMBup}\]
        \[\BPD{\BoldMBdown \BoldBUP} \quad \BPD{\BoldMBdown \BoldX}\quad\BPD{\BoldMBdown \BoldMBup} 
        \quad \BPD{\BoldBDown \BoldBUP} \quad \BPD{\BoldBDown \BoldX}\quad\BPD{\BoldBDown \BoldMBup} \quad \BPD{\BoldX \BoldBUP} \quad \BPD{\BoldX \BoldX}\quad \BPD{\BoldX \BoldMBup}\]
        \[\BPD{\BoldBDown \BoldBUP} \quad \BPD{\BoldBDown \BoldX}\quad \BPD{\BoldBDown \BoldMBup}\]    
        
               \vspace{5pt}
        after the deletion/merging algorithm
        \caption{Local configuration of tiles before and after the deletion/merging algorithm.}\label{vbisv-09}
 \end{figure}

In Figure \ref{vbisv-09}, we observe that $\Phi_j$ does not change the route of the pipe (which is drawn in boldface) that we   consider. So the proof is complete.
\end{proof}

\begin{remark}
Let $v$ be a subword of $w$. Clearly, we can directly generalize the notion of removable pipes to any marked  reduced  pipe dream $P$ of  $v$. Accordingly, we can define the   map  $\Phi_j$ for any removable pipe $j$ in $P$.    
\end{remark}

\begin{proposition}\label{p3.3}
Let $P\in \mathrm{MRPD}(w)$, and pipe  $j$ and pipe $j'$ ($j\neq j'$) be two removable pipes in $P$. Then pipe $j'$ is also removable in $\Phi_j(P)$. Moreover, $\Phi_j$ and $\Phi_{j'}$ commute, that is,
\[
\Phi_{j'}(\Phi_j(P))=\Phi_{j}(\Phi_{j'}(P)).
\]
\end{proposition}

\begin{proof}
We first explain  that pipe $j'$ is also  removable  in $\Phi_j(P)$. 
According to  the proof of Proposition \ref{lem:well_defined}, 
the route of pipe $j'$ remains  unchanged after the deletion/merging algorithm. 
Clearly, for pipe $j'$ in $\Phi_j(P)$, the conditions  (i) and (iii) in the definition of a removable pipe are satisfied. It remains to check the condition  in (ii). Suppose that $(x,y)$ is a crossing tile in $\Phi_j(P)$ traversed by pipe $j'$.
We need to check that the tile $(x-1,y)$ (if any) in $\Phi_j(P)$ is not an unmarked  bumping tile  $\BPD{\B} $. 
Let $B$ denote the tile (if any)  immediately above $(x,y)$ in $P$. Since pipe $j'$ is removable in $P$, it follows  that $B$ is a crossing tile or a marked bumping tile. If $B$ is not traversed by pipe $j$ in $P$, then $B$ will not be removed  after the deletion procedure, and thus  $(x-1,y)=B$ is not $\BPD{\B} $. 
Now consider the case when $B$ is traversed by pipe $j$. In this case, $B$ must be a crossing. Since pipe $j$ is removable in $P$, the tile in $P$ immediately above $B$ (if any), denoted $B'$, cannot be $\BPD{\B} $. This implies that in  $\Phi_j(P)$, we have $(x-1,y)=B'$ which is not $\BPD{\B} $.  So condition (ii) is satisfied for  pipe $j'$ in 
$\Phi_j(P)$, concluding that pipe $j'$ is   removable  in $\Phi_j(P)$.

Now we see that $\Phi_{j'}(\Phi_j(P))$ is obtained from $P$   by deleting all tiles or semi-tiles traversed by  pipes $j$ or $j'$, and then performing the merging operation, independent of the order of $j$ and $j'$. So we have  $\Phi_{j'}(\Phi_j(P))=\Phi_{j}(\Phi_{j'}(P))$.
\end{proof}

\begin{remark}\label{aghh-678}
Let $P\in \mathrm{MRPD}(w)$, and pipe  $j$ be a removable pipe of $P$. Using the similar analysis to that in the proof of Proposition \ref{p3.3}, we can show that if pipe $j'$ in $P$ is not removable, then it is also not removable in  $\Phi_j(P)$.
    
\end{remark}

By Proposition \ref{p3.3} and Remark  \ref{aghh-678}, we are able to repeatedly apply the map $\Phi_j$ to erase    all the removable pipes of a marked reduced pipe dream $P\in \mathrm{MRPD}(w)$, regardless of the order in which the pipes are deleted. As a result, we obtain a core marked  reduced  pipe dream corresponding to a subword of $w$.  More concretely, assume that $P$ has $k$ removable pipes labeled $j_1,\ldots, j_k$, and $v$ is the  subword   by deleting the entries $j_1,\ldots, j_k$ of $w$. Then
\[
\Phi(P)=\Phi_{j_k}(\cdots \Phi_{j_2}(\Phi_{j_1}(P))\cdots)
\]
is a core marked reduced pipe dream of $v$. 
We shall call $\Phi$ the reduction map.
Figure \ref{fig:reduction} gives an illustration of the  construction of $\Phi$. 

\begin{figure}[h]
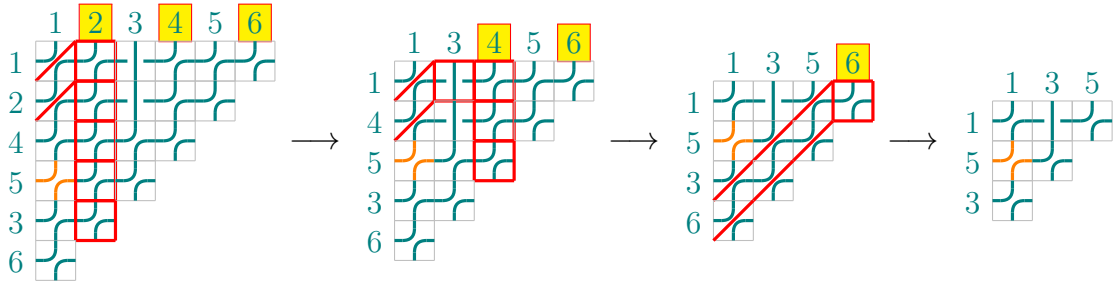

    \centering
    \[
\BPD[1.25pc]{
\M{}\M{1}\M{\fcolorbox{red}{yellow}{2}}\M{3}\M{\fcolorbox{red}{yellow}{4}}\M{5}\M{\fcolorbox{red}{yellow}{6}}\\   
\M{1}\delBDown\delB\X\B\B\B \\
\M{2}\delBUP\delB\X\B\B \\
\M{4}\B\delB\B\B \\
\M{5}\MB\delB\B \\
\M{3}\B\delB \\
\M{6}\B \\
}\longrightarrow
\BPD[1.25pc]{
\M{}\M{1}\M{3}\M{\fcolorbox{red}{yellow}{4}}\M{5}\M{\fcolorbox{red}{yellow}{6}}\\   
\M{1}\delBDown\delX\delB\B\B \\
\M{4}\delBUP\X\delB\B \\
\M{5}\MB\B\delB \\
\M{3}\B\B \\
\M{6}\B \\ 
}\longrightarrow
\BPD[1.25pc]{
\M{}\M{1}\M{3}\M{5}\M{\fcolorbox{red}{yellow}{6}}\\   
\M{1}\B\X\delBDown\delB \\
\M{5}\MB\delBDown\delBUP \\
\M{3}\delBDown\delBUP \\
\M{6}\delBUP\\
}\longrightarrow
\BPD[1.25pc]{
\M{}\M{1}\M{3}\M{5}\\   
\M{1}\B\X\B \\
\M{5}\MB\B \\
\M{3}\B \\
\\
}
\]
    \caption{Illustration of the reduction map $\Phi$.}
    \label{fig:reduction}
\end{figure}
Just as for a permutation, write $\mathrm{CMRPD}(v)$ for the set of core marked reduced pipe dreams of a subword $v$. Now we have a map
\[
\Phi\colon \mathrm{MRPD}(w)\rightarrow \bigsqcup_{v\leq w} \mathrm{CMRPD}(v),
\]
where $v$ runs over all subwords of $w$. A crucial fact is that the map $\Phi$ is an injection, which  will be shown in the next subsection.

\subsection{The augmenting algorithm}

Let $u<v\leq w$ be two subwords of $w$ such that $u$ is obtained from $v$ by deleting  an entry $j$. 
For $P\in \mathrm{MRPD}(v)$, let $P'=\Phi_j(P)$ which is a marked reduced pipe dream of $u$. We shall explain    how to  uniquely recover  $P$    from $P'$.

Write $v=v_1v_2\cdots v_m$ and $u=v_1\cdots \hat{v}_i\cdots v_m$, where $\hat{v}_i=j$ means deletion. Moreover, assume that there are $\ell$ entries in $v$ that are less than $j$. 

\begin{itemize}
    \item[(1)] First, for  $1\leq k \leq  \ell$, we insert a crossing tile $\BPD{\X}$ or two consecutive  semi-tiles \(\BPD{\delBDown\\ \delBUP\\}\) (lying between the red lines) in the $k$-th column of $P'$ as follows.   Start with $k=1$. 
Look at the tile $(i-1,k)$ in row $i-1$ and   column $k$. 
\begin{itemize}
    \item[$\bullet$] If $(i-1,k)$ is not an unmarked bumping tile \(\BPD{\B}\), then shift all the tiles in column $k$ and below row $i-1$ downward by one unit, and insert a crossing tile in the blank  position $(i,k)$.  
   Otherwise (that is, $(i-1,k)$ is  \(\BPD{\B}\)), then shift the  lower-right  simi-tile of $(i-1,k)$, along with all the tiles in column $k$ below it,  downward by one unit, and then insert two consecutive semi-tiles into the resulting  blank space.  In the latter case, we need to reset $i:=i-1$.  Do the same for $i$ and $k:=k+1$, and terminate  until $k=\ell+1$.     
\end{itemize}

\item[(2)] It remains to recover   the $(\ell+1)$-th  column  that has been deleted from $P$ (which is the column where pipe $j$ enters). Suppose that after the above operation, the pointer variable $i$ is equal to $i_0$. We shift all the tiles, which lie in or strictly to the right of column $\ell+1$ in $P'$, to the right by one unit, and then insert $i_0-1$ crossing tiles followed by $m-\ell$ bumping tiles in column $\ell+1$.  

The above procedure  is best illustrated  by reading  Figure \ref{fig:the_operation_of_Phi_{w,i,u}} in reverse. 

\end{itemize}
We denote the above operation by $\Psi_j$. That is, if $P'=\Phi_j(P)$, then 
\[
\Psi_j(P')=P. 
\]

Implementing  the whole reduction map gives 
\[
Q=\Phi(P)=\Phi_{j_k}(\cdots \Phi_{j_2}(\Phi_{j_1}(P))\cdots),
\]
where $Q$ is a core pipe dream of the  subword obtained from  $w$ by deleting $j_1,\ldots, j_k$. 
Then we have 
\[
P=\Psi_{j_1}(\cdots \Psi_{j_{k-1}}(\Psi_{j_k}(Q)\cdots). 
\]
This procedure, denoted $\Psi$, is called  the augmenting map.

Collecting the above, 
we conclude that 
\begin{proposition}\label{inj-12}
 The reduction  map 
 \[
\Phi\colon \mathrm{MRPD}(w)\rightarrow \bigsqcup_{v\leq w} \mathrm{CMRPD}(v),
\]
is an injection. Moreover, for any $P\in \mathrm{MRPD}(w)$, we have $P=\Psi(\Phi(P))$. 
\end{proposition}

For a subword $v\leq w$ of $w$, denote 
 \[
d_v{(\beta)}= \sum_{P\in \operatorname{CMRPD}(v)} \beta^{   \operatorname{mbt}(P) }.
  \]
Since $\Phi$ preserves   marked bumping tiles, as an immediate  application of Proposition  \ref{inj-12}, we obtain the following upper bound for $\Upsilon_w{(\beta)}$. 

\begin{corollary}
For any permutation $w$, we have the following coefficient-wise inequality 
     \[
\Upsilon_w{(\beta)}\leq_{\mathrm{coeff}} \sum_{v\leq w}\,    d_v{(\beta)}.
  \]
\end{corollary}

Let us take an example to illustrate the injection $\Phi$. 
 In Figure   \ref{fig:Red_inj}, we apply the map $\Phi$ to all the marked reduced pipe dreams of  $w=1423$. In this example, the map $\Phi$ is not surjective. Actually, as shown in  Figure   \ref{fig:Red_inj-2},  there are two additional  core marked reduced pipe dreams of the subword $143$  which do not appear as images  in Figure   \ref{fig:Red_inj}. 
\begin{figure}[h t]
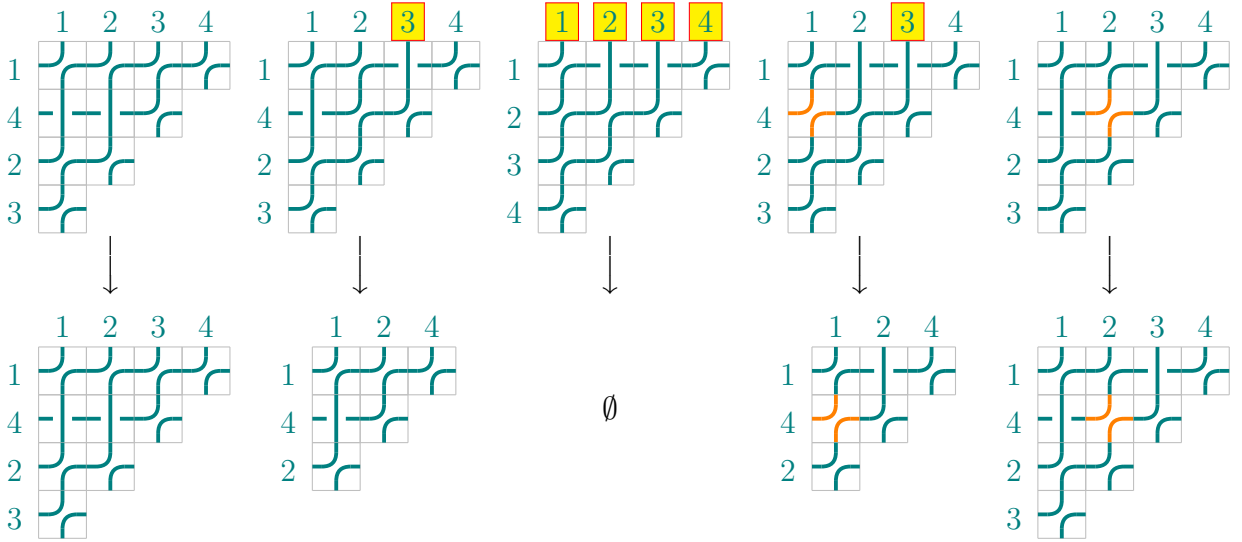

    \centering
    
    \[
\overset{\BPD[1.5pc]{
\M{}\M{1}\M{2}\M{3}\M{4}\\   
\M{1}\B\B\B\B \\
\M{4}\X\X\B \\
\M{2}\B\B \\
\M{3}\B \\
}}{\underset{\BPD[1.5pc]{
\M{}\M{1}\M{2}\M{3}\M{4}\\   
\M{1}\B\B\B\B \\
\M{4}\X\X\B \\
\M{2}\B\B \\
\M{3}\B \\
}}{\Big\downarrow}}
  \overset{\BPD[1.5pc]{
\M{}\M{1}\M{2}\M{\fcolorbox{red}{yellow}{3}}\M{4}\\  
\M{1}\B\B\X\B \\
\M{4}\X\B\B \\
\M{2}\B\B \\
\M{3}\B \\
}}{\underset{\BPD[1.5pc]{
\M{}\M{1}\M{2}\M{4}\\   
\M{1}\B\B\B \\
\M{4}\X\B \\
\M{2}\B \\
}}{\Big\downarrow}}
  \overset{ \BPD[1.5pc]{
\M{}\M{\fcolorbox{red}{yellow}{1}}\M{\fcolorbox{red}{yellow}{2}}\M{\fcolorbox{red}{yellow}{3}}\M{\fcolorbox{red}{yellow}{4}}\\   
\M{1}\B\X\X\B \\
\M{2}\B\B\B \\
\M{3}\B\B \\
\M{4}\B \\
}}{\underset{\raisebox{-40pt}{$\emptyset$}}{\Big\downarrow}}
\overset{\BPD[1.5pc]{
\M{}\M{1}\M{2}\M{\fcolorbox{red}{yellow}{3}}\M{4}\\  
\M{1}\B\X\X\B \\
\M{4}\MB\B\B \\
\M{2}\B\B \\
\M{3}\B \\
}}{\underset{\BPD[1.5pc]{
\M{}\M{1}\M{2}\M{4}\\   
\M{1}\B\X\B \\
\M{4}\MB\B \\
\M{2}\B \\
}}{\Big\downarrow}} 
\overset{\BPD[1.5pc]{
\M{}\M{1}\M{2}\M{3}\M{4}\\   
\M{1}\B\B\X\B \\
\M{4}\X\MB\B \\
\M{2}\B\B \\
\M{3}\B \\
}}{\underset{\BPD[1.5pc]{
\M{}\M{1}\M{2}\M{3}\M{4}\\   
\M{1}\B\B\X\B \\
\M{4}\X\MB\B \\
\M{2}\B\B \\
\M{3}\B \\
}}{\Big\downarrow}}   
    \]
    \caption{The reduction map applied to marked reduced pipe dreams of $1423$.}

    \label{fig:Red_inj}
\end{figure}

\begin{figure}[h t]
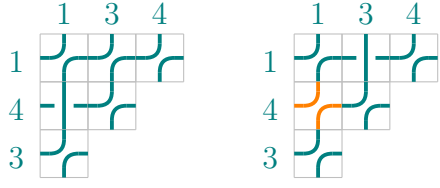

    \centering

\[\BPD[1.5pc]{
\M{}\M{1}\M{3}\M{4}\\   
\M{1}\B\B\B \\
\M{4}\X\B \\
\M{3}\B \\
}\ \ \ \ \ 
\BPD[1.5pc]{
\M{}\M{1}\M{3}\M{4}\\   
\M{1}\B\X\B \\
\M{4}\MB\B \\
\M{3}\B \\
}\]

    \caption{Core marked reduced pipe dreams that are  not contained in $\Phi(\mathrm{MRPD}(1423))$.}

    \label{fig:Red_inj-2}
\end{figure}

\begin{remark}\label{ji-087}
Of course, we can perform $\Psi_j$ on any marked reduced pipe dream $P$ of a subword in a permutation $w$. However, in   general, the resulting (marked) line diagram $\Psi_j(P)$ is not necessarily   a  legal marked reduced pipe dream. The possible flaw is that there may exist two pipes in $\Psi_j(P)$ crossing  more than once. It is not hard  to check that  if this happens, one of the two pipes must be the newly  added  pipe $j$ since the intersection state of the pipes in $P$ will not be changed after the action of $\Psi_j$.

For example, in Figure    \ref{fig:placeholder}, we illustrate the resulting line diagrams by applying $\Psi_2$ to the marked reduced pipe  dreams of the subword $143$ given in Figure \ref{fig:Red_inj-2}. As can be seen, the newly added pipe $2$  intersects pipe  $3$  twice  in each of the images. 
\end{remark}

\begin{figure}[h]
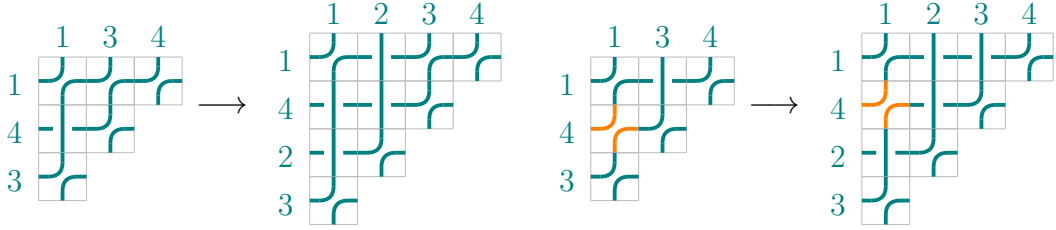

    \centering
    \[\BPD[1.5pc]{
    \M{}\M{1}\M{3}\M{4}\\   
    \M{1}\B\B\B \\
    \M{4}\X\B \\
    \M{3}\B \\
    }\longrightarrow
    \BPD[1.5pc]{
    \M{}\M{1}\M{2}\M{3}\M{4}\\   
    \M{1}\B\X\B\B \\
    \M{4}\X\X\B \\
    \M{2}\X\B\\
    \M{3}\B \\
    }\quad\BPD[1.5pc]{
    \M{}\M{1}\M{3}\M{4}\\   
    \M{1}\B\X\B \\
    \M{4}\MB\B \\
    \M{3}\B \\
    }\longrightarrow
    \BPD[1.5pc]{
    \M{}\M{1}\M{2}\M{3}\M{4}\\   
    \M{1}\B\X\X\B \\
    \M{4}\MB\X\B \\
    \M{2}\X\B\\
    \M{3}\B \\
    }\]
    \caption{Applying   $\Psi_2$ to the marked reduced pipe dreams in Figure \ref{fig:Red_inj-2}. }
    \label{fig:placeholder}
\end{figure}

\subsection{Proofs of   Theorems \ref{thm:Grothendieck_main} and  \ref{main-1}}

Our goal in this subsection is to prove  that the  reduction map $\Phi$ is a bijection when we restrict $w$ to a permutation  avoiding  the $1423$ pattern. 
This allows us to finish the proofs of   Theorems \ref{thm:Grothendieck_main} and  \ref{main-1}.

\begin{theorem}\label{bij-al}
Let $w\in S_n$ be a permutation avoiding the  $1423$ pattern. Then the map
     \[
\Phi\colon \mathrm{MRPD}(w)\rightarrow \bigsqcup_{v\leq w} \mathrm{CMRPD}(v)
\]
is a bijection. 
\end{theorem}

\begin{proof}
By Proposition \ref{inj-12}, the map $\Phi$ is an injection. So we need to show that when $w$ does not contain the $1423$ pattern,  $\Phi$ is a surjection. 

Let $u<v\leq w$ be two subwords of $w$ such that $u$ is obtained from $v$ by deleting the entry $j$.  Since $w$ avoids the $1423$ pattern,   both $u$ and $v$ avoid the $1423$ pattern.  Assume that $P$ is a given marked reduced pipe dream of $u$. As explained in Remark \ref{ji-087}, we may perform  the augmenting algorithm $\Psi_j$ on $P$, resulting in a line   diagram, denoted    $Q=\Psi_j(P)$, corresponding to  the subword $v$.   

\begin{itemize}
    \item[] {\bf Claim.} Any two pipes in $Q$ can cross at most once. 
\end{itemize}
By Remark \ref{ji-087}, we only need to prove that pipe $j$ intersects pipe $k$ ($k\neq j$)  at most once in $Q$. We shall give a proof  by contradiction. 

Suppose that pipe $j$ and pipe  $k$ intersect more than once. 
According to the construction  of \(\Psi_{j}\), the crossing tile where pipe $j$ traverses vertically  appears only in the column where $j$ enters. That means that $j<k$, and moreover these two pipes cross exactly twice, and  the  northeast  crossing position   occurs in the column where pipe  $j$ enters. Assume that the southwest crossing position   occurs at $(x,y)$. Then, by the construction of $\Psi_j$, the tile  $(x-1,y)$ is either  a crossing tile or a marked bumping tile. We discuss these two cases separately, which can be well understood by means of  the illustration in Figure  \ref{ill-avjo}  where the tile $(x,y)$ is shaded.

\begin{itemize} 
    \item[] Case 1. The tile  $(x-1,y)$ is a crossing.  Note that pipe $j$ traverses horizontally through $(x,y)$, and pipe $k$  traverses vertically through $(x-1,y)$. Suppose the pipe traversing horizontally through $(x-1,y)$ is  labeled  $\ell$. By Remark \ref{ji-087},  pipe $k$ and pipe $\ell$   cross exactly once in $Q$, indicating that $\ell>k$. If pipe $\ell$ and pipe $j$ cross more than  once, then  pipe $\ell$ exits on the left boundary  between   pipe $j$ and  pipe $k$. If this happens, we  replace pipe  $k$ by pipe $\ell$ and repeat the above analysis, getting a new pipe  $k$ exiting from the left boundary closer to   pipe $j$. Assuming that we have chosen the closest one, it follows that  pipe $\ell$ and pipe $j$ have to  cross  only once. Therefore pipe $\ell$ exits above pipe $j$, and thus the subword $\ell j k$ forms a $312$ pattern of $v$.

Another observation is that   pipe $k$ traverses vertically through the tile $(x-1,y)$, and so it will yield a bumping tile $(x',y)$ which is  in  column $y$ and above row $x-1$. Then, among all the bumping tiles contained in the rectangle $x'\times y$,  locate the top-leftmost one. We see that the upper-left pipe in this tile, denoted   pipe $i$, enters from some column to the left of column $y$ and exits from a row above $x'$. Now we obtain that the subword  $i\ell j k$  forms a $1423$ pattern of $v$, arriving at a contradiction.

    \item[] Case 2.   The tile  $(x-1,y)$ is a marked bumping tile.  We also suppose that the upper-left pipe in  $(x-1,y)$ is pipe $\ell$. Using completely the same  analysis as in Case 1, we may assume that pipe $\ell$ exits from a row  above pipe $j$. Similarly, there is a bumping tile formed  by pipe $\ell$ which is in column $y$ and above row $x-1$. We can then find a pipe $i$ with $i<j$ that exits from the left boundary above row $x-1$.  Thus the subword  $i\ell j k$ also forms a $1423$ pattern of $v$, leading to a contradiction.  
\end{itemize}
The arguments in the above two cases allow us to conclude the  proof of the {\bf claim}. 

Now we see that $Q$ is a well-defined marked reduced pipe dream of $v$. Therefore, for any subword $v\leq w$ and any $Q\in \mathrm{CMRPD}(v)$, we may repeatedly apply the augmenting map to $v$, eventually yielding a marked reduced pipe dream  $P\in \mathrm{MRPD}(w)$ such that $\Phi(P)=Q$. This confirms that $\Phi$ is  a surjection. 
\end{proof}

\begin{figure}[h]
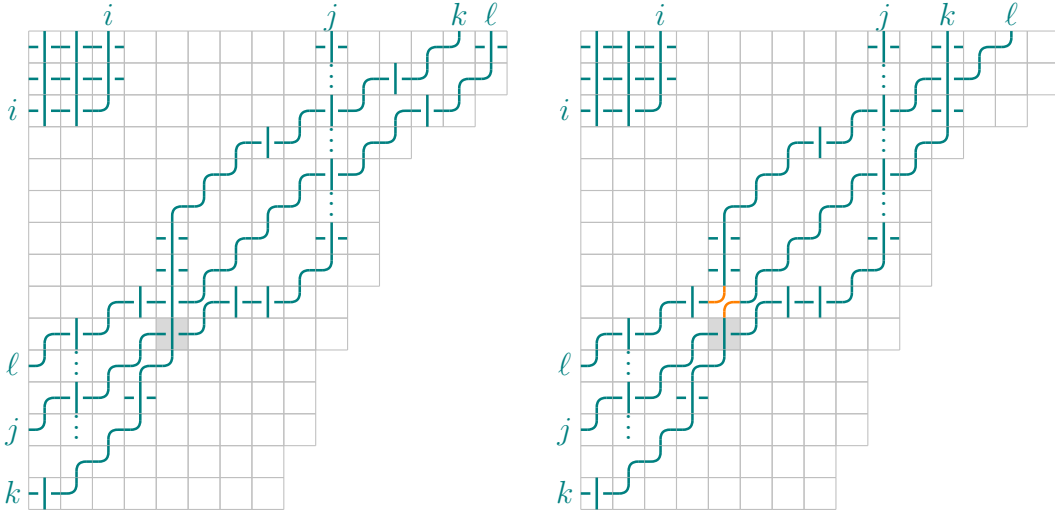

    \centering
    
\[\BPD[1pc]{
\M{}\M{}\M{}\M{i}\M{}\M{}\M{}\M{}\M{}\M{}\M{j}\M{}\M{}\M{}\M{k}\M{\ell}\\   
\M{}\X\X\X\O\O\O\O\O\O\X\O\O\F\J\X \\
\M{}\X\X\X\O\O\O\O\O\O\Lb{\vdots}\F\X\J\F\J\\
\M{i}\X\X\J\O\O\O\O\O\F\X\J\F\X\J\\
\M{}\O\O\O\O\O\O\F\X\J\Lb{\vdots}\F\J\\
\M{}\O\O\O\O\O\F\J\O\F\X\J\\
\M{}\O\O\O\O\F\J\O\F\J\Lb{\vdots}\O\\
\M{}\O\O\O\O\X\O\F\J\O\X\O\\
\M{}\O\O\O\O\X\F\J\O\F\J\O\\
\M{}\O\O\F\X\X\B\X\X\J\O \\
\M{}\F\X\J\F\grayX\J\O\O\O\O \\
\M{\ell}\J\Lb{\vdots}\F\B\J\O\O\O\O \\
\M{}\F\X\J\X\O\O\O\O\O \\
\M{j}\J\Lb{\vdots}\F\J\O\O\O\O\O \\
\M{}\O\F\J\O\O\O\O\O \\
\M{k}\X\J\O\O\O\O\O\O \\
}\quad
\BPD[1pc]{
\M{}\M{}\M{}\M{i}\M{}\M{}\M{}\M{}\M{}\M{}\M{j}\M{}\M{k}\M{}\M{\ell}\M{}\\   
\M{}\X\X\X\O\O\O\O\O\O\X\O\X\F\J\O \\
\M{}\X\X\X\O\O\O\O\O\O\Lb{\vdots}\F\X\J\O\O\\
\M{i}\X\X\J\O\O\O\O\O\F\X\J\X\O\O\\
\M{}\O\O\O\O\O\O\F\X\J\Lb{\vdots}\F\J\\
\M{}\O\O\O\O\O\F\J\O\F\X\J\\
\M{}\O\O\O\O\F\J\O\F\J\Lb{\vdots}\O\\
\M{}\O\O\O\O\X\O\F\J\O\X\O\\
\M{}\O\O\O\O\X\F\J\O\F\J\O\\
\M{}\O\O\F\X\MB\B\X\X\J\O \\
\M{}\F\X\J\F\grayX\J\O\O\O\O \\
\M{\ell}\J\Lb{\vdots}\F\B\J\O\O\O\O \\
\M{}\F\X\J\X\O\O\O\O\O \\
\M{j}\J\Lb{\vdots}\F\J\O\O\O\O\O \\
\M{}\O\F\J\O\O\O\O\O \\
\M{k}\X\J\O\O\O\O\O\O \\
}
\]
\caption{Illustration for the proof of Theorem  \ref{bij-al}. }
\label{ill-avjo}
 
\end{figure}

We are now in a position to provide the  proofs of   Theorems \ref{thm:Grothendieck_main} and  \ref{main-1}. 
For any subword $v$, let 
  \[
 d_v{(\beta)}= \sum_{P\in \operatorname{CMRPD}(v)} \beta^{ \mathrm{mbt}(P)}.
  \]
Theorem \ref{thm:Grothendieck_main}
follows from    the following statement.

\begin{theorem}\label{hf-008}
Let $w$ be a permutation avoiding $1423$. Then we have $c_w(\beta)= d_w{(\beta)}$.    
\end{theorem}

\begin{proof}
Notice that \[d_v(\beta)=d_{\mathrm{perm}(v)}(\beta).\]
 By Theorem \ref{bij-al}, we obtain that 
    \[
\Upsilon_w(\beta)=\sum_{v\leq w} d_v(\beta)=\sum_{v\leq w} d_{\mathrm{perm}(v)}(\beta).
\]  
This implies 
\[
d_w{(\beta)}=\sum_{\emptyset\leq v\leq w} (-1)^{|w|-|v|} \Upsilon_{\mathrm{perm}(v)}{(\beta)},
\]
which exactly  agrees with $c_w(\beta)$. 
\end{proof}

We can use  similar arguments to  give a proof of Theorem \ref{main-1}. Fix a permutation $w$ and a subword $u\leq w$. For any subword $u\leq v\leq w$, 
let $\mathrm{CMRPD}(u,v)$ denote the set of marked reduced  pipe dreams $P$ of $v$ such that for any entry $j$ in $v$ but not in $u$, pipe $j$ is not removable in $P$. When $u=\emptyset$, it is clear that $\mathrm{CMRPD}(\emptyset,v)=\mathrm{CMRPD}(v)$.
Define 
\[
d_{u,v}(\beta)=\sum_{P\in \mathrm{CMRPD}(u,v)}\beta^{\mathrm{mbt}(P)}.
\]
Evidently, we have $d_{\emptyset, v}=d_v$ and $d_{u,v}=0$ for $u\not\leq v$.

\begin{proof}[Proof of Theorem \ref{main-1} ]
In  the reduction process, if we require that all pipes whose labels belong to $u$ cannot be deleted, then we obtain an injection 

 \[
\Phi_u\colon \mathrm{MRPD}(w)\rightarrow \bigsqcup_{u\leq v\leq w} \mathrm{CMRPD}(u,v).
\] 
Of course, when $w$ does not contain the pattern $1423$, by the proof of Theorem \ref{bij-al}, the above  injection will be   a bijection. 
So we have 
  \[
\Upsilon_w(\beta)=\sum_{u\leq v\leq w} d_{u,v}(\beta).
\]  
Notice that the poset structure of the interval $[u,w]$ is isomorphic to the Boolean lattice $[\emptyset, w']$ where $w'$ is the subword of $w$ by deleting the entries in $u$. 
So, applying   inclusion-exclusion, 
we deduce that 
\[
\sum_{u\leq v\leq w} (-1)^{|w|-|v|} \Upsilon_{\mathrm{perm}(v)}{(\beta)}=d_{u,w}(\beta) \in \mathbb{Z}_{\geq 0}[\beta],
\]
as desired. 
\end{proof}

\section{Concluding remarks}
\label{last-sec}

We lastly  discuss some points that are possibly of interest for further research. As aforementioned,
Me\'sz\'aros and Tanjaya \cite{meszaros2022inclusion} proved Conjecture 1.1 for permutations avoids simultaneously
1423 and 1432. In this setting, a combinatorial interpretation for $c_w$ in terms of certain
diagrams, which are less than or equal to the Rothe diagram of $w$ in the Gale order, is given
in \cite[Theorem]{meszaros2022inclusion}. See \cite[Section 3.3]{meszaros2022inclusion} for the concrete construction of such diagrams. On
the other hand, by Theorem \ref{hf-008} (in the case $\beta=0$), the value $c_w$ can be interpreted as the
number of core reduced pipe dreams of $w$.

\begin{problem} 
Does there exist a bijection between these two combinatorial objects which have been  developed  to count $c_w$ (where $w$ avoids $1423$   and    $1432$)?
\end{problem}

Knutson and Udell  \cite{knutson} introduced the hybrid pipe dream model for Schubert polynomials, which includes the pipe dream model and the bumpless pipe dream model as extremal cases. 
A hybrid pipe dream is defined in terms of  tilings of the $n\times n$ grid such that each row is either of type C, using the classical tiles 
\[\BPD{\X}, \BPD{\H},\BPD{\J} , \BPD{\F} ,\BPD{\B},\BPD{\O},\text{ but not } \BPD{\I},
\] or of  type B, using the bumpless tiles .
\[\BPD{\X}, \BPD{\H},\BPD{\Z} , \BPD{\L} ,\BPD{\I},\BPD{\O},\text{ but not } \BPD{\B}.
\]
It was proved in \cite{knutson} that hybrid pipe dreams give a formula for Schubert polynomials by fixing a  tiling type for each row, thus yielding a total of $2^n$ models for generating $\mathfrak{S}_w(x)$. 
It reproduces the classic pipe dream (resp., bumpless pipe dream) model when each row is of  type C (resp., B). 

The hybrid pipe dream model might serve as  a broader platform   which 
could provide a common framework for the techniques developed for
 pipe dreams  in this paper and for bumpless pipe dreams in \cite{ALCO_2025__8_3_745_0}.

\begin{problem}
Is it possible to explore a similar approach for hybrid pipe dreams?
\end{problem}

\footnotesize{

\textsc{\{Haojun Bai, Feng Gu, Peter L. Guo, Jiaji Liu\} Center for Combinatorics, Nankai University, LPMC, Tianjin 300071, P.R. China}

\vspace{10pt}
{\it
Email address: \tt  haojunbai@mail.nankai.edu.cn, fgu@mail.nankai.edu.cn, lguo@nankai.edu.cn, 

1120240006@mail.nankai.edu.cn}


\begin{thebibliography}{100}

\bibitem{anderson2026computationsamplingschubertspecializations}
D. Anderson, G. Panova and L. Petrov, Computation and sampling for Schubert specializations, arXiv: 2603.20104, 2026. 

\bibitem{bergeron1993rc}
N. Bergeron and  S.C. Billey, RC-graphs and Schubert polynomials, Experiment. Math. 2 (1993), no. 4, 257--269.



\bibitem{BS}
V. Buciumas and T. Scrimshaw, Double Grothendieck polynomials and colored lattice models, Int. Math. Res. Not. IMRN 2022, no. 10, 7231--7258. 

\bibitem{CS}
J. Chou and  L. Setiabrata, 
Asymptotically maximal Schubitopes, arXiv:2512.04053v1, 2025. 


\bibitem{ALCO_2025__8_3_745_0}
H. Dennin, Pattern bounds for principal specializations of $\beta ${-Grothendieck} polynomials, Algebr. Combin. 8 (2025), no. 3, 745--763.


\bibitem{2018Schubert}
A. Fink, K. M{\'e}sz{\'a}ros and A. St. Dizier,
Schubert polynomials as integer point transforms of generalized permutahedra,
Adv. Math. 332 (2018), 465--475.

\bibitem{fink2021zero}
A. Fink, K. M{\'e}sz{\'a}ros and A. St. Dizier,
Zero-one Schubert polynomials,
Math. Z. 297 (2021),   1023--1042.


\bibitem{fomin1994grothendieck}
S. Fomin and  A.N. Kirillov, Grothendieck polynomials and the Yang-Baxter equation, Center for Discrete Mathematics and Theoretical Computer Science (DIMACS), Piscataway, NJ, 2007, 183--189.

\bibitem{FS}
S. Fomin and R.P. Stanley, Schubert polynomials and the nil-Coxeter algebra, Adv. Math. 103 (1994),  196--207.



\bibitem{gao2021principal}
Y. Gao, Principal specializations of Schubert polynomials and pattern containment, European J. Combin. 94 (2021), Paper No. 103291, 12 pp.

\bibitem{guo2024schubert}
P.L. Guo and Z. Lin, Schubert polynomials and patterns in permutations, Math. Z., to appear.

\bibitem{HPSW}
Z. Hamaker,  O. Pechenik, D.E. Speyer and A. Weigandt,  Derivatives of Schubert polynomials and proof of a determinant conjecture of Stanley,  Algebraic  Combin. 3 (2020), 301--307.

\bibitem{JK}
G. James and A. Kerber, The representation theory of the symmetric group, Encyclopedia of Mathematics
and its Applications, vol. 16, Addison-Wesley Publishing Co., Reading, Mass., 1981, With a foreword by P. M. Cohn,
With an introduction by Gilbert de B. Robinson.



\bibitem{KNUTSON2004161}
A. Knutson and E. Miller, Subword complexes in Coxeter groups, Adv. Math. 184 (2004), no. 1, 161--176.

\bibitem{knutson2005grobner}
A. Knutson and E. Miller,  Gr{\"o}bner geometry of Schubert polynomials, Ann. of Math. (2) 161 (2005), no. 3, 1245--1318.

\bibitem{knutson}
A. Knutson and G. Udell, Interpolating between classic and bumpless pipe dreams, S{\'e}m. Lothar. Combin. 89B (2023), Art. 89, 12 pp.

\bibitem{LLS}
T. Lam, S.J. Lee and M. Shimozono, Back stable Schubert calculus, Compos. Math. 157 (2021),
no. 5, 883--962. 

\bibitem{LLS-2}
T. Lam, S.J. Lee and  M. Shimozono, Back stable K-theory Schubert calculus, Int. Math. Res. Not. 24 
(2023), 21381--21466.

\bibitem{lascoux1989fonctorialite}
A. Lascoux and  M.-P. Sch\"utzenberger,
Polyn\^omes de Schubert,
C. R. Acad. Sci. Paris Sér. I Math. 294 (1982),   447--450.

\bibitem{LS-Gro}
A. Lascoux and M.-P. Sch\"utzenberger, Structure de Hopf de l’anneau de cohomologie et de l’anneau de Grothendieck d’une vari\'et\'e de drapeaux, C. R. Acad. Sci.
Paris S\'er. I Math. 295 (1982), 629--633.


\bibitem{lenart2006grothendieck}
C. Lenart, S. Robinson and F. Sottile,  Grothendieck polynomials via
permutation patterns and chains in the Bruhat order, Amer. J. Math. 128 (2006), no. 4, 805--848.

\bibitem{Mac}
I.G. Macdonald, Notes on Schubert Polynomials, Laboratoire de combinatoire et d'informatique math\'ematique (LACIM), Universit\'e du Qu\'ebec \'a Montr\'eal, Montreal, 1991.


\bibitem{meszaros2022inclusion}
K. M{\'e}sz{\'a}ros and A. Tanjaya, Inclusion-exclusion on Schubert polynomials, Algebr. Combin. 5 (2022), no. 2, 209--226.

\bibitem{MPP-1}
A.H. Morales, I. Pak  and G. Panova, Asymptotics of principal evaluations of
 Schubert polynomials for layered permutations, Proc. Amer. Math. Soc.  147 (2019), 1377--1389.

\bibitem{MPP-2} A.H. Morales, I. Pak and G. Panova, Hook formulas for skew shapes IV. Increasing tableaux and 
factorial Grothendieck polynomials, J. Math. Sci. 261 (2022), 630--657.

 


\bibitem{morales2025grothendieck}
A.H. Morales, G. Panova, L. Petrov and D. Yeliussizov, Grothendieck shenanigans: permutons from pipe dreams via integrable probability, Adv. Math. 480 (2025), part C, Paper No. 110510, 63 pp.

\bibitem{stanley}
R.P. Stanley,
Some Schubert shenanigans,
arXiv:1704.00851v2, 2017.

\bibitem{weigandt2018schubert}
A. Weigandt, Schubert polynomials, 132-patterns, and Stanley’s conjecture, Algebr. Combin. 1 (2018), no. 4, 415--423.

\bibitem{Wei}
A. Weigandt, Bumpless pipe dreams and alternating sign matrices, J. Combin. Theory Ser. A 182 (2021),
105470. 




\bibitem{Zhang-1}
N. Zhang, 
Principal specializations of Schubert polynomials, multi-layered permutations and asymptotics, Adv. Appl. Math. 163 (2025),    102806, 19 pp.

\end{thebibliography}
\end{document}